\documentclass[oneside,11pt]{article}
\usepackage[utf8]{inputenc}
\usepackage[T1]{fontenc}
\usepackage{amsmath,amsfonts,amssymb,amsthm,amscd,tipa,array,mathtools,longtable,tikz-cd,booktabs}
\usepackage[english]{babel}
\usepackage{fancyhdr}
\usepackage{indentfirst}
\usepackage{color}
\usepackage{graphicx}
\usepackage[a4paper,top=4.5cm,bottom=3.5cm]{geometry}
\usepackage[all]{xy}
\usepackage{hyperref}
\hypersetup{hidelinks}
\usepackage{tikz-cd}
\usetikzlibrary{arrows}

\renewcommand{\epsilon}{\varepsilon}

\newcommand{\C}{\mathbb{C}}
\newcommand{\hc}{\textbf H^2_\C}
\newcommand{\R}{\mathbb{R}}

\newcommand{\zzz}{\textbf t}
\newcommand{\www}{\textbf s}
\newcommand{\zzc}[1]{t_{#1}}
\newcommand{\wwc}[1]{s_{#1}}
\newcommand{\conf}[1]{\textcircled{#1}}
\newcommand{\br}[3]{br_{#1}(#2,#3)}
\newcommand{\cone}{(\alpha,\alpha,\theta,\phi)}
\newcommand{\ctwo}{(\pi+\theta-\alpha,\alpha,2\alpha-\pi,\pi+\theta+\phi-2\alpha)}
\newcommand{\cthree}{(\alpha,\pi+\theta-\alpha,\theta,\phi)}
\newcommand{\vertr}[3]{\ensuremath{\textbf v_{#1,#2,#3}}} 

\DeclareMathOperator{\im}{Im}

\DeclareMathOperator{\id}{Id}

\DeclareMathOperator{\area}{Area}

\newcommand{\gA}[1]{\ensuremath{A_{#1}}} 
\newcommand{\gR}[1]{\ensuremath{R_{#1}}} 
\newcommand{\gRR}[1]{\ensuremath{B_{#1}}} 

\theoremstyle{plain}
\newtheorem{theo}{Theorem}[section]
\newtheorem{lemma}[theo]{Lemma}
\newtheorem{prop}[theo]{Proposition}

\theoremstyle{definition}

\newtheorem{rk}[theo]{Remark}

\title{Fundamental polyhedra for all Deligne-Mostow lattices in $PU(2,1)$}
\author{I. Pasquinelli}

\begin{document}

\maketitle

\begin{abstract}
In this work we will build a fundamental domain for Deligne-Mostow lattices in $PU(2,1)$ with 2-fold symmetry, which complete the whole list of Deligne-Mostow lattices in dimension 2. 
These lattices were introduced by Deligne and Mostow, in \cite{mostow2}, \cite{mostow3} and \cite{delignemostow} using monodromy of hypergeometric functions and have been reinterpreted by Thurston in \cite{thurston} as authomorphisms on a sphere with cone singularities. 
Following his approach, Parker in \cite{livne}, Boadi and Parker in \cite{boadiparker} and Pasquinelli in \cite{irene} built a fundamental domain for the class of lattices with 3-fold symmetry, i.e. when three of five cone singularities have same cone angle.
Here we extend this construction to the asymmetric case, where only two of the five cone points on the sphere have same cone angle, so to have a fundamental domain for each commensurability class of Deligne-Mostow lattices in $PU(2,1)$. 
\end{abstract}


\section{Introduction}\label{sec:intro}

Deligne-Mostow lattices first appeared in \cite{delignemostow}, \cite{mostow2} and \cite{mostow3}.
They arise as monodromy of hypergeometric functions, a construction that dates back to Picard, Lauricella and others.
More precisely, they start with a \emph{ball $N$-tuple} $\mu=(\mu_1, \dots \mu_N)$, i.e. a set of $N$ real numbers between 0 and 1 such that $\sum \mu_i=2$.
Then they deduce a sufficient condition on $\mu$ for the monodromy group to be discrete, called condition $\Sigma$INT. 
This improved the sufficient condition called INT and introduced by Picard.

In \cite{thurston}, Thurston reinterpreted these lattices in terms of cone metrics on the sphere. 
First he considers a sphere with $N$ cone singularities of cone angles $\theta_i $ between 0 and $2\pi$ satisfying the discrete Gauss-Bonnet formula (i.e. $\sum \alpha_i =4\pi$, where $\alpha_i=2\pi-\theta_i$ are the curvatures at the cone points). 
Then he proves that the moduli space of such cone metrics with prescribed cone angles and area 1 has a complex hyperbolic structure of dimension $N-3$. 
Considering the automorphisms of the sphere swapping cone points (and their squares), he gets some conditions on the cone points to obtain a lattice. 
Thurston's criterion corresponds to the $\Sigma$INT condition when taking $\mu_i=2\pi \alpha_i $.
In \cite{sadayoshi}, Kojima proved that the two constructions are equivalent.

Combining the works of Deligne, Mostow and Thurston already mentioned with the work of Sauter (see \cite{sauter}), one gets a finite and exhaustive list of ball $N$-tuples $\mu $ that give rise to a lattice using this construction (see also \cite{dmbook}). 
This includes nine ball $5$-tuples and one ball $6$-tuple not satisfying the condition $\Sigma$INT, but commensurable to a monodromy group satisfying $\Sigma$INT.
Any other value gives a non-discrete quotient. 
In this work we will concentrate on the ball $5$-tuples in Deligne-Mostow and Thurston's works and we will study the lattices in $PU(2,1)$ obtained.
All of these lattices have an extra symmetry given by some of the $\mu_i $'s having the same value. 
In Thurston's approach, this means that some of the cone points on the sphere have same cone angle. 
In particular, the lattices will have either 2-fold or 3-fold symmetry (i.e. they will have 2 or 3 cone points with same cone angle respectively). 
The latter case has been analysed in many different papers. 
First, in \cite{type2}, Deraux, Falbel and Paupert built a fundamental domain for the lattices of second type.
In \cite{livne}, using a different method, Parker built a fundamental domain for a class of these lattices, called Livné lattices (or lattices of third type).    
The lattices of first type were treated by Boadi and Parker in \cite{boadiparker}, using the same procedure as in \cite{livne}. 
Later, in \cite{irene}, I explained how to use Parker's method to describe a single polyhedron that, appropriately modified, gives a fundamental domain for all lattices with 3-fold symmetry, including the cases already treated and the lattices of fourth type.  

The goal of this paper is to show how to adapt Parker's construction in order to build a fundamental domain for the remaining Deligne-Mostow lattices, namely those with 2-fold symmetry. 
In the first part of the paper, we will forget about the symmetries at all and we will give a completely general construction that is valid whatever the initial cone points are. 
The construction consists in parametrising the cone metric and showing geometrically that Thurston's theorem holds, giving explicitly the Hermitian form that determines the complex hyperbolic structure. 
Then we will introduce the \emph{moves}, which are maps on the sphere corresponding to swapping two cone points, i.e. applying half Dehn twist along a curve containing two cone points or corresponding to a full Dehn twist. 
These are automorphisms of the sphere when the cone points which are swapped have same cone angle. 
Here we will also consider maps that swap cone points with different cone angle.
This means that we land on a new cone metric after applying the move. 
Moreover, we will show how one can build a polyhedron by studying what happens when pairs of cone points approach until they coalesce. 
We want to remark that this is completely general and can be built even if the cone angles we started from do not give a lattice. 
Then, in the 3-fold symmetry case the polyhedron is actually a fundamental domain for the lattices, when starting from the right set of cone singularities. 
In the 2-fold symmetry case this polyhedron is a building block for the fundamental domain, which will consist of the union of three copies of this polyhedron, each for a different ordering of the cone points. 
How to take these three copies is the topic of the second part of the paper, together with the proof that the new polyhedron built is the fundamental domain we wanted. 

More precisely, Section \ref{sec:intro} is the present introduction. 

Section \ref{sec:config} considers a generic cone metric on the sphere (without any symmetry). 
To parametrise it, we cut along curves passing through the cone points and develop the metric on a plane in a polygonal form. 
One can recover the cone metrics by glueing the associated sides of the polygon back together. 
The parameters will be related to the sides of the polygon, which we use to give a set of projective coordinates. 
We will then describe the moves and the polyhedron. 
In the last part of the section, we will describe the two sets of coordinates that we will use to make the definition of the polyhedron clearer, we will analyse its cells and study its combinatorics. 

In Section \ref{sec:2f} we will specialise to the 2-fold symmetry case. 
First we will introduce the lattices we will be working with, listing the possible sets of cone angles we will be starting from. 
This is the original list from the works of Deligne and Mostow which can be found, for example, at the end of \cite{thurston} or in \cite{mostow3}.
Then we will build a new polyhedron as the union of three copies of the polyhedron described in Section \ref{sec:2fold}. 
It will be described using three sets of coordinates.
At the end of the section we will describe its sides and use the moves of Section \ref{sec:moves} to construct the side pairing maps that we need for Poincaré polyhedron theorem. 

Section \ref{sec:mainthm} is dedicated to our main theorem, which states that the polyhedron constructed (up to certain modifications) is indeed a fundamental domain for the lattices with 2-fold symmetry. 
This is proved using Poincaré polyhedron theorem and in this section we will show that all conditions in the theorem are satisfied. 
Specifically we will prove the tessellation condition for each ridge (2-dimensional facet) of the polyhedron. 
The theorem also gives us a presentation of the lattices, with the side pairings as generators. 
In the presentation the relations are given by cycle transformations which are complex reflections with certain parameters related to the lattice as their order.
When an order is positive, we have a complex reflection with respect to a complex line, when it is negative we have a complex reflection in a point, while when it is $\infty$ we have a parabolic element and a fixed point on the boundary. 
The last two cases are related to the modifications of the polyhedron that we mentioned. 
In fact, when one of the parameters is negative or infinite, a ridge collapses to a single point, on the boundary when the parameters is infinite. 
In Section \ref{sec:main} we also explain in details why this happens. 
In particular, when a particular one of the parameters is not positive and finite, we need to consider a different configuration (similar to the one in Section \ref{sec:config}, but not quite the same), which is described in Section \ref{sec:kneg}. 
Such explicit description of the polyhedron also allows us to calculate the orbifold Euler characteristic of the polyhedron, as the sum (with alternating signs with the dimension of the facets) of the order of stabiliser of one element for each orbit of facets.
Then we calculate the volume of the quotient, which is a multiple of the orbifold Euler characteristic. 
Remark that the volume we calculated is coherent with the commensurability theorems we know for these lattices (see, for example, page 15 of \cite{survey}) and the known volumes of the lattices. 

Some of our proofs are very similar to the ones in \cite{livne}, \cite{boadiparker} and \cite{irene}. 
When the exact same proof can be used, we will omit to rewrite it. 

This research was supported by a fellowship under the program "FY2015 JSPS Postoctoral Fellowship (short-term) for North American and European Researchers",  while visiting Tokyo Institute of Technology and by a Doctoral EPSRC grant, awarded by Durham University. 
I would like to thank Professor Parker for his constant inspiration and encouragement. 
I would also like to thank Professor Kojima and his laboratory for hosting me during my fellowship, for supporting my work and for the useful discussions. 
 
\section{The general construction}\label{sec:config}

This section follows Section 2 of \cite{livne}, Section 2 of \cite{boadiparker} and Section 4 of \cite{irene}. 
It generalises this procedure to when there is no symmetry given by cone points having same cone angle. 
In the first part of this section we will show how to parametrise cone metrics on the sphere of area 1 and five cone singularities. 
This is a generic construction and does not depend on whether the cone angles we choose give a lattice or not, nor on whether the cone points have same angle or not. 
The only restriction on the cone angles in this case is for the Hermitian form we obtain to have the required signature. 
In the second part of the section we will show how to build a polyhedron in the moduli space starting from the cone metrics and using the coordinates we introduced.
Finally we will describe some maps we will use, in the spirit of the moves in previous works. 
In the case of lattices with 3-fold symmetry treated in previous works, the polyhedron so constructed is a fundamental domain for the group generated by the moves, which are symmetries of the polyhedron. 
In the case of 2-fold symmetry treated in the Section \ref{sec:2f}, it will be a building block for the fundamental domain. 

\subsection{Configurations space}

Consider a cone metric on the sphere with cone points of angles $\theta_0, \theta_1, \theta_2, \theta_3$ and $\theta_4$ with $0<\theta_i<2\pi$ and $\sum (2\pi-\theta_i)=4\pi$ (discrete Gauss-Bonnet formula).
Since we have 5 cone singularities, a priori the lattices are described by 5 parameters. 

The discrete Gauss-Bonnet formula guarantees that the value of the fifth angle is determined by the previous four.
To prescribe the cone angles we will use the parameters 
\begin{align}\label{eq:defparam}
\alpha&=\frac{\theta_1}{2}, & \beta&=\frac{\theta_2}{2}, 
& \theta&=\frac{\theta_2}{2}+\frac{\theta_3}{2}-\pi, &
\phi&=\frac{\theta_0}{2}+\frac{\theta_1}{2}-\pi.
\end{align}
They have a geometric meaning which is made clear in Figure \ref{fig:thetapar}.
Then we will denote a cone metric with these cone angles as $(\alpha, \beta, \theta, \phi)$.
By definition of the parameters, we are considering a flat sphere with 5 cone singularities of angles 
\[
(2(\pi+\phi-\alpha), 2\alpha, 2\beta, 2(\pi+\theta-\beta), 2(\pi-\theta-\phi)).
\]
As one can see in the upper-left-hand side of Figure \ref{fig:thetapar}, the order of the angles is given by starting in the lower left corner and continuing counter-clockwise.
So the angle $\theta_i$ is the cone angle of the cone point $v_i$ for $i=0,1,2,3$ and  $\theta_4$ is the cone angle of the cone point $v_*$.

We now fix the cone angles (so fix a configuration $(\alpha, \beta, \theta, \phi)$) and want to parametrise all possible positions of the cone points on the sphere.
Let us first consider the easier case of when the five cone singularities are along the equator of the sphere.
Then one can cut along a geodesic passing through $v_0, v_1, v_2, v_3$ and $v_*$ in order and open up the figure obtained, getting the upper-right-hand side of Figure \ref{fig:thetapar}. 
The sphere is then the shaded region $\Pi$ in the figure, obtained by considering two big triangles $T_3$ and $T_{-3}$ and removing from it two copies of a smaller triangle $T_2$ and $T_{-2}$ and two copies of another small triangle $T_1$ and $T_{-1}$. 
Since in this case all the possible variations in the cone metric are the possible distances of the various points, one just needs a parameter describing the length of the sides of each of the three triangles to have the whole configuration (hence the cone metric) completely determined. 
Since all the angles are determined by the cone angles, it is enough to parametrise one side of each of the triangles and to do this we will use the parameters $\zzc 1$, $\zzc 2$ and $\zzc 3$ shown in the picture. 

To parametrise the general case, it is enough to allow the three real parameters $\zzc 1$, $\zzc 2$ and $\zzc 3$ to be complex (i.e. not to be aligned), since this encodes the fact that two pieces of the geodesic might not divide the cone angle they share in two equal angles. 

For more details on this description one can see Section 2.1 of \cite{livne}, Section 2.1 of \cite{boadiparker} and Section 4.1 of \cite{irene}.

\begin{figure}
\centering
\includegraphics[width=0.85\textwidth]{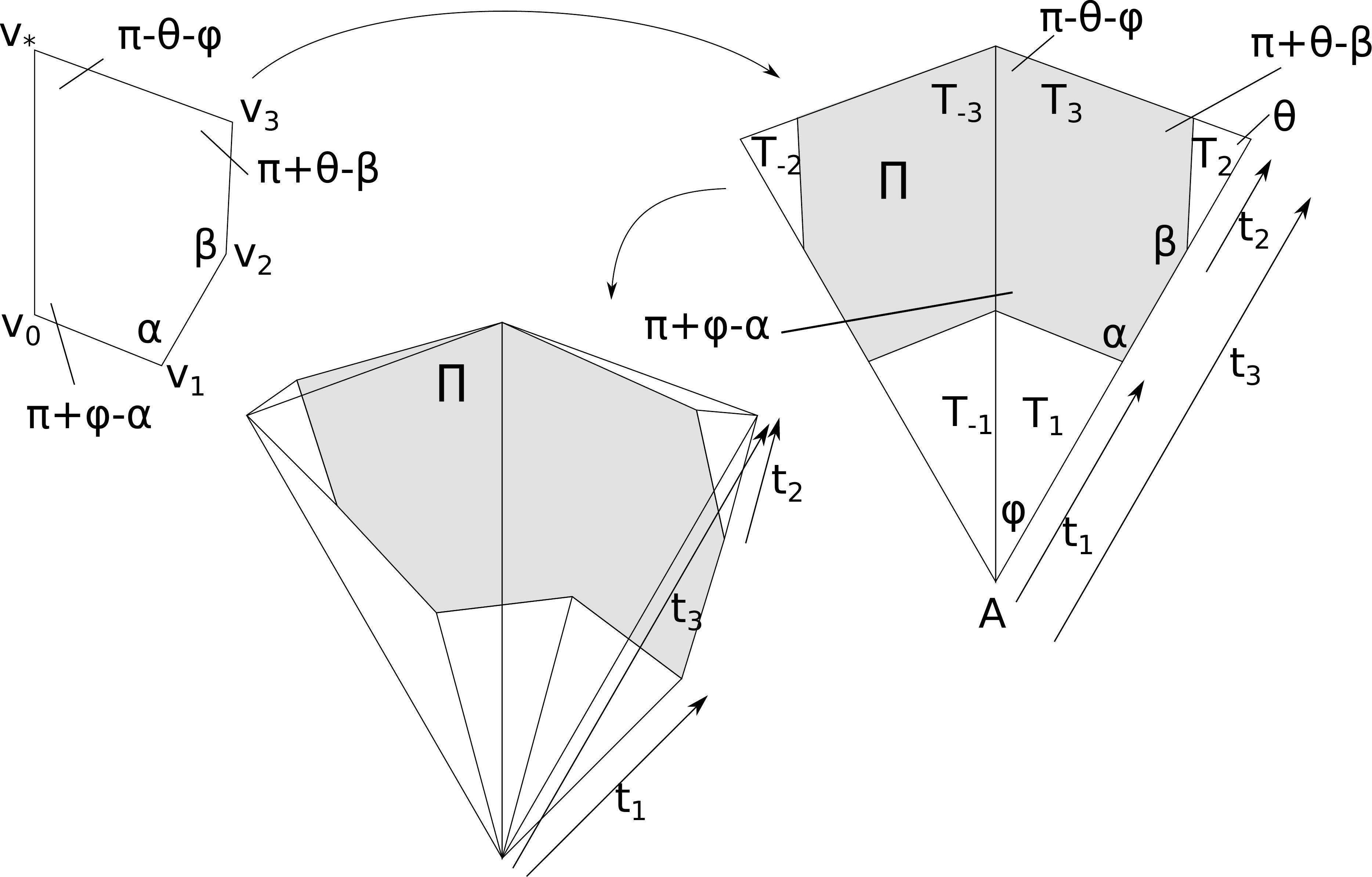}
\begin{quote}\caption{The configuration $(\alpha, \beta, \theta, \phi)$.\label{fig:thetapar}} \end{quote}
\end{figure}


Since we are interested in the cone metrics up to rescaling, we will choose to assume that $\zzc 3=1$.
Remark that this is one of the possible normalisations, different from asking from the area to be 1 (like in \cite{thurston}).

As proved by Thurston, the set of cone metrics of this type has a complex hyperbolic structure.
The Hermitian form for the complex hyperbolic structure is given by the area of the cone metric.
The area of the octagon $\Pi$ is given by 
\begin{equation}\label{eq:area}
\area \Pi= \frac{\sin \theta \sin \phi}{\sin(\theta+\phi)}|\zzc 3|^2 
-\frac{\sin \theta \sin\beta}{\sin(\beta-\theta)}|\zzc 2|^2 
-\frac{\sin \phi \sin \alpha}{\sin(\alpha-\phi)}|\zzc 1|^2
\end{equation}
So one could write the Hermitian form in matrix form as 
\[
H=
\begin{bmatrix}
-\frac{\sin \phi \sin \alpha}{\sin (\alpha -\phi)}& 0 & 0 \\
0 & -\frac{\sin \theta \sin \beta}{\sin(\beta- \theta)} & 0 \\
0 & 0 & \frac{\sin \phi\sin \theta}{\sin(\theta +\phi)}
\end{bmatrix},
\]
and say that
\[
\area \Pi=\textbf t^* H \textbf t.
\]

Since this only makes sense if the area of $\Pi $ (and so the area of the cone metric) is positive, we can write 
\[
\hc=\{ \textbf z \colon \langle \textbf z, \textbf z \rangle = \textbf z^* H \textbf z > 0 \}.
\]

\subsection{Moves}\label{sec:moves}

In this section we will introduce some maps that will have a key role in the following sections, since their compositions will be the generators of the lattices with 2-fold symmetry. 
They generalise the maps used in \cite{livne}, \cite{boadiparker} and \cite{irene}.

The move $\gR 1$ exchanges the two cone points $v_2$ and $v_3$ with their cone angles, while $\gR 2$ exchanges $v_1$ and $v_2$. 
Since the moves change the values of our parameters, we will denote the move as $\gR i(\alpha, \beta, \theta, \phi)$ to say that $\gR i \colon (\alpha, \beta, \theta, \phi) \mapsto (\alpha', \beta', \theta',\phi')$, unless the angles of the configuration we start from is obvious. 
This means, for example, that when composing two maps $T(\alpha, \beta, \theta,\phi)$ and $S(\alpha, \beta, \theta,\phi)$, we need to consider that the second map is applied to the new angles, so we are doing the composition
\begin{equation}\label{eq:composition}
S(\alpha', \beta', \theta',\phi') \circ T(\alpha, \beta, \theta,\phi)
\end{equation}
because $(\alpha, \beta, \theta, \phi) \xmapsto{T} (\alpha', \beta', \theta',\phi') \xmapsto{S} (\alpha'', \beta'', \theta'',\phi'')$.
Similarly, when calculating inverses we have 
\begin{equation}\label{eq:inverse}
[T(\alpha, \beta, \theta,\phi)]^{-1}=T^{-1}(\alpha', \beta', \theta',\phi'),
\end{equation}
since $T \colon (\alpha, \beta, \theta, \phi) \mapsto (\alpha', \beta', \theta',\phi')$ and $T^{-1} \colon (\alpha', \beta', \theta',\phi') \mapsto (\alpha, \beta, \theta, \phi)$.

The matrix of $\gR 1(\alpha, \beta, \theta, \phi)$ is now obtained from the equations $v_0'=v_0$, $v_*'=v_*$, $v_1'=v_1$, $v_3'=v_2$ and $v_{-2}'=v_3$, where the $v_i$'s are the coordinates in the $(\alpha, \beta, \theta, \phi)$ configuration and the $v_i'$'s in the $(\alpha', \beta', \theta',\phi')$ configuration. 
The matrix of $\gR 1$ is then
\begin{equation}\label{eq:R1matrix}
\gR 1(\alpha, \beta, \theta, \phi)=
\begin{bmatrix}
1 & 0 & 0 \\
0 & e^{i \theta}\frac{\sin \beta}{\sin (\beta-\theta)} & 0 \\
0 & 0 & 1 \\
\end{bmatrix}.
\end{equation}

\begin{figure}
\centering
\includegraphics[width=1\textwidth]{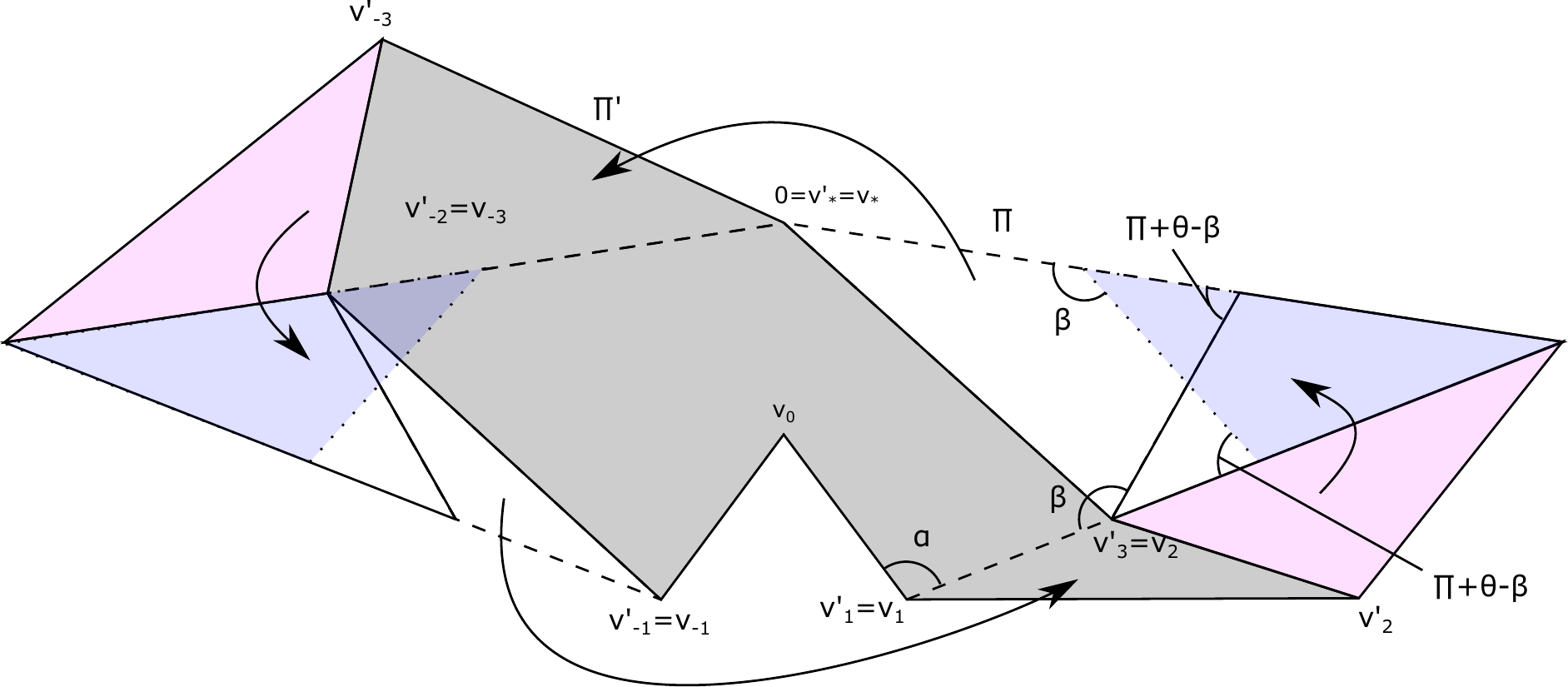}
\begin{quote}\caption{The move $\gR 1$. \label{fig:R1}} \end{quote}
\end{figure}

Similarly, one can find the matrix of $\gR 2$ by simultaneously solving the equations $v_0'=v_0$, $v_*'=v_*$, $v_2'=v_1$, $v_{-1}'=v_{-2}$ and $v_3'=v_3$ and gets: 
\begin{align} \label{eq:R2matrix}
&\gR 2(\alpha, \beta, \theta, \phi)=\frac{1}{\sin(\theta+\alpha-\beta)\sin (\phi+\beta-\alpha)} \nonumber\\
&=\begin{bmatrix}
\sin \alpha \sin\theta'e^{i(\alpha-\phi)} & 
\sin(\alpha-\phi) \sin\theta'e^{i\alpha} & 
-\sin(\alpha-\phi) \sin\theta'e^{i\alpha} \\
\sin(\beta-\theta) \sin\phi'e^{i\beta} & 
\sin\phi'\sin \beta e^{i(\beta-\theta)} & 
-\sin(\beta-\theta) \sin\phi'e^{i\beta} \\
\sin(\theta+\phi)\sin \alpha e^{i\beta} & 
\sin(\theta+\phi)\sin \beta e^{i\alpha} & 
A\\
\end{bmatrix},
\end{align}
with $\phi'=\phi+\beta-\alpha$ and $\theta'=\theta+\alpha-\beta$ and
\begin{align}\label{eq:Ainstd}
A&=\sin \theta \sin\phi'- \sin(\theta+\phi)\sin \beta e^{i\alpha} \\
&=\sin \phi \sin\theta'- \sin(\theta+\phi)\sin \alpha e^{i\beta} \nonumber \\
&\sin \theta \sin \phi \cos(\alpha-\beta)-\sin\theta\cos \phi\sin \alpha e^{i \beta}-\cos \theta\sin\phi\sin\beta e^{i\alpha}. \nonumber
\end{align}

The third move $\gA 1$ is exactly as in previous papers because it starts and lands in the same configuration and its matrix is 
\begin{equation}\label{eq:A1matrix}
\gA 1(\alpha,\beta,\theta,\phi)=\begin{bmatrix}
e^{2i \phi} & 0 & 0 \\
0 & 1 & 0 \\
0 & 0 & 1
\end{bmatrix}.
\end{equation}

\begin{figure}[t]
\centering
\includegraphics[width=1\textwidth]{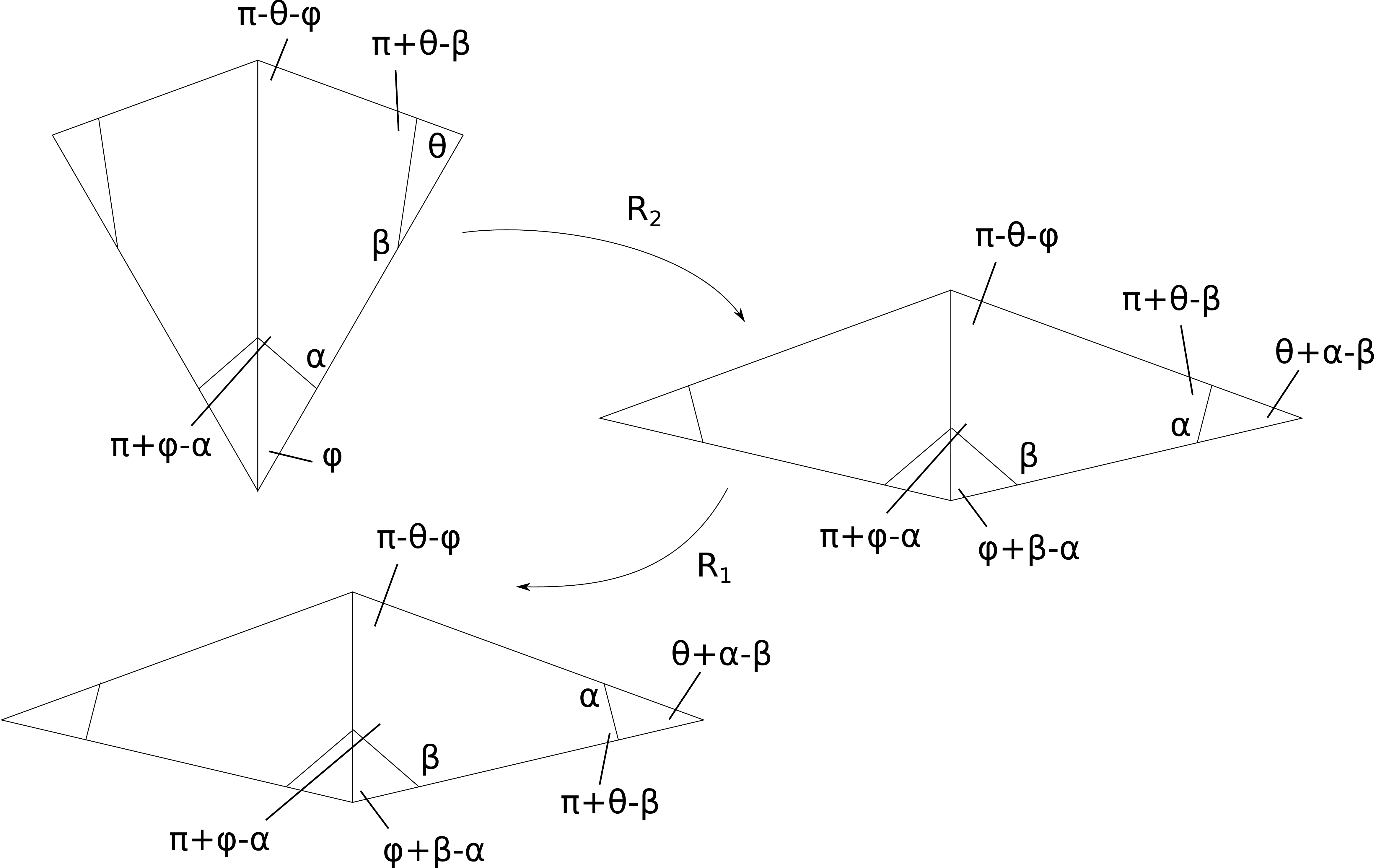}
\begin{quote}\caption{The action of $P$ on the angles. \label{fig:Paction}} \end{quote}
\end{figure}

We now want to compose the two matrices to calculate $P=\gR 1\gR 2$ and $J=P\gA 1$.
As we already mentioned, after applying the first transformation the angles change. 
Looking at Figure \ref{fig:Paction}, one can deduce that
\begin{align} \label{eq:Pmatrix}
P(\alpha, \beta, \theta,\phi)&=\gR 1(\alpha', \beta', \theta',\phi')\gR 2(\alpha, \beta, \theta,\phi) \nonumber\\
&=\gR 1(\beta, \alpha, \theta+\alpha -\beta,\phi+\beta-\alpha)\gR 2(\alpha, \beta, \theta,\phi) \nonumber\\
&=\begin{bmatrix}
\sin \alpha \sin\theta'e^{i(\alpha-\phi)} & 
\sin(\alpha-\phi) \sin\theta'e^{i\alpha} & 
-\sin(\alpha-\phi) \sin\theta'e^{i\alpha} \\
\sin \alpha \sin\phi'e^{i(\alpha+\theta)} & 
\frac{\sin\phi'\sin \beta \sin \alpha}{\sin(\beta-\theta)} e^{i\alpha} & 
-\sin \alpha \sin\phi'e^{i(\alpha+\theta)} \\
\sin(\theta+\phi)\sin \alpha e^{i\beta} & 
\sin(\theta+\phi)\sin \beta e^{i\alpha} & 
A\\
\end{bmatrix},
\end{align}
where, as before, $\phi'=\phi+\beta-\alpha$, $\theta'=\theta+\alpha-\beta$ and $A$ is as in \eqref{eq:Ainstd}.

On the other hand, $J=P\gA 1$ is easier to calculate, since $\gA 1$ does not change the type of the configuration. 
So 
\begin{align} \label{eq:Jmatrix}
&J(\alpha, \beta, \theta,\phi)=P(\alpha, \beta, \theta,\phi)\gA 1(\alpha, \beta, \theta,\phi) \nonumber\\
&=\begin{bmatrix}
\sin \alpha \sin\theta'e^{i(\alpha+\phi)} & 
\sin(\alpha-\phi) \sin\theta'e^{i\alpha} & 
-\sin(\alpha-\phi) \sin\theta'e^{i\alpha} \\
\sin \alpha \sin\phi'e^{i(\alpha+\theta+2\phi)} & 
\frac{\sin\phi'\sin \beta \sin \alpha}{\sin(\beta-\theta)} e^{i\alpha} & 
-\sin \alpha \sin\phi'e^{i(\alpha+\theta)} \\
\sin(\theta+\phi)\sin \alpha e^{i(\beta+2\phi)} & 
\sin(\theta+\phi)\sin \beta e^{i\alpha} & 
A\\
\end{bmatrix},
\end{align}
where again $\phi'=\phi+\beta-\alpha$, $\theta'=\theta+\alpha-\beta$ and $A$ is as in \eqref{eq:Ainstd}.

We remark that if we define a second set of coordinates as $\www =P^{-1} \zzz$ (as we will do later), the action of $\gR 2$ is equivalent to applying $\gR 1$ on the $\www$-coordinates. 
In other words, $\gR 2=P\gR 1 P^{-1}=\gR 1 \gR 2 \gR 1 \gR 2 ^{-1} \gR 1^{-1}$, which is equivalent to prove the braid relation
\[
\gR 1\gR 2\gR 1=\gR 2\gR 1\gR 2.
\]
Again, to calculate this composition, we need to record how the configuration changes when applying the matrices so we need to prove that the following diagram holds
\[
\begin{tikzcd}
(\alpha, \beta,\theta,\phi) \arrow{r}{\gR 1^{-1}} \arrow{d}[swap]{\gR 2}
& (\alpha, \pi+\theta-\beta,\theta,\phi) \arrow{d}{\gR 2^{-1}} \\
(\beta, \alpha, \alpha+\theta-\beta, \beta+\phi-\alpha)
& (\pi+\theta-\beta, \alpha, \alpha+\beta -\pi, \pi+\theta+\phi-\alpha-\beta) \arrow{d}{\gR 1} \\
(\beta, \pi+\theta-\beta, \alpha+\theta-\beta, \beta+\phi-\alpha) \arrow{u}{\gR 1}
& (\pi+\theta-\beta, \beta,\alpha+\beta -\pi, \pi+\theta+\phi-\alpha-\beta) \arrow{l}{\gR 2}
\end{tikzcd}
\]
which is easy to verify by simple calculation.

\subsection{The polyhedron}\label{sec:genpoly}

\subsubsection{Complex lines and vertices}

We want to study the metric completion of the moduli space. 
This means that we want to see what happens when two cone points get closer and closer until they coalesce.
The complex line $L_{ij}$ is obtained when $v_i$ and $v_j$ coalesce, for $i,j \in \{0,1,2,3,*\}$. 
Its normal vector will be denoted as $\textbf n_{ij}$.
They have equations as follow:

\begin{longtable}{|c|c|}
\hline
$L_{ij}$ & Equations in terms of the $\zzz$-coordinates \\
\hline
$L_{*0}$ & $\zzc 1=\frac{\sin(\alpha-\phi)\sin \theta}{\sin \alpha \sin (\theta+\phi)}$ \\
$L_{*1}$ & $\zzc 1=e^{-i\phi} \frac{\sin \theta}{\sin(\theta+\phi)}$ \\
$L_{*2}$ & $\zzc 2=e^{i\theta} \frac{\sin \phi}{\sin(\theta+\phi)} $ \\
$L_{*3}$ & $ \zzc 2=\frac{\sin(\beta-\theta)\sin \phi}{\sin \beta \sin(\theta+\phi)}$ \\
$L_{01}$ & $\zzc 1=0$ \\
$L_{02}$ & $\frac{\sin \alpha}{\sin(\alpha -\phi)}e^{i \phi}\zzc 1+ \zzc 2=1$ \\
$L_{03}$ & $\frac{\sin \alpha}{\sin(\alpha -\phi)}e^{i \phi}\zzc 1+ e^{-i\theta}\frac{\sin \beta}{\sin(\beta-\theta)}  \zzc 2=1$ \\
$L_{12}$ & $\zzc 1+\zzc 2=1$ \\
$L_{13}$ & $\zzc 1+ e^{-i\theta}\frac{\sin \beta}{\sin(\beta-\theta)}  \zzc 2=1$ \\
$L_{23}$ & $\zzc 2=0$ \\
\hline 
\caption{The equations defining the complex lines of two cone points collapsing.}
\label{table:lineseq}
\end{longtable}

The vertices of the polyhedron are obtained by intersecting two of these complex lines (i.e. by making two pairs of cone points coalesce) and they have coordinates as follows. 

\begin{longtable}{|c|c|c|c|}
\hline
Lines & $\zzz_k$ & $\zzc 1$ & $\zzc 2$ \\
\hline
$L_{01} \cap L_{23}$ & $\zzz_1$ 
& 0 
& 0  \\
$L_{03} \cap L_{12}$ & $\zzz_2$ 
& $\frac{\sin(\alpha-\phi)(\sin(\beta-\theta)-e^{-i\theta}\sin\beta)}
{e^{i\phi} \sin \alpha \sin (\beta-\theta)-e^{-i\theta}\sin\beta \sin(\alpha-\phi)}$ 
& $\frac{e^{i\alpha}\sin(\beta-\theta)\sin \phi }
{e^{i\phi}\sin\alpha \sin(\beta-\theta)-e^{-i\theta}\sin\beta \sin(\alpha-\phi)} $ \\
$L_{*0} \cap L_{23}$ & $\zzz_3$ & 
$\frac{\sin(\alpha-\phi) \sin \theta}{\sin \alpha \sin (\theta+\phi)}$ 
& 0 \\
$L_{*0} \cap L_{12}$ & $\zzz_4$ 
& $\frac{\sin(\alpha-\phi) \sin \theta}{\sin \alpha \sin (\theta+\phi)}$ 
& $\frac{\sin(\alpha+\theta) \sin \phi}{\sin \alpha \sin(\theta+\phi)}$ \\
$L_{*0} \cap L_{13}$ & $\zzz_5$ 
& $\frac{\sin(\alpha-\phi) \sin \theta}{\sin \alpha \sin (\theta+\phi)}$ 
& $e^{i\theta}\frac{\sin(\alpha+\theta)\sin(\beta-\theta) \sin \phi}{\sin \alpha \sin \beta \sin(\theta+\phi)}$ \\
$L_{*1} \cap L_{23}$ & $\zzz_6$ 
& $e^{-i\phi} \frac{\sin \theta}{\sin(\theta+\phi)}$ 
& 0 \\
$L_{*1} \cap L_{02}$ & $\zzz_7$ 
& $e^{-i\phi} \frac{\sin \theta}{\sin(\theta+\phi)}$ 
& $\frac{\sin (\alpha-\theta-\phi) \sin \phi}{\sin(\alpha -\phi) \sin (\theta+\phi)}$ \\
$L_{*1} \cap L_{03}$ & $\zzz_8$ 
& $e^{-i\phi} \frac{\sin \theta}{\sin(\theta+\phi)}$ 
& $e^{i\theta}\frac{\sin (\alpha-\theta-\phi)\sin(\beta-\theta) \sin \phi}
{\sin(\alpha -\phi) \sin \beta \sin (\theta+\phi)}$ \\
$L_{*3} \cap L_{01}$ & $\zzz_9$ 
& 0 
& $\frac{\sin (\beta - \theta)\sin \phi}{\sin \beta \sin(\theta+\phi)}$\\
$L_{*3} \cap L_{12}$ & $\zzz_{10}$ 
& $\frac{\sin (\beta + \phi) \sin \theta }{\sin \beta \sin(\theta+\phi) }$ 
& $\frac{\sin (\beta - \theta)\sin \phi}{\sin \beta \sin(\theta+\phi)}$\\
$L_{*3} \cap L_{02}$ & $\zzz_{11}$ 
& $e^{-i\phi}\frac{\sin(\alpha -\phi) \sin (\beta + \phi) \sin \theta}{\sin \alpha \sin \beta \sin(\theta+\phi)}$ 
& $\frac{\sin (\beta - \theta)\sin \phi}{\sin \beta \sin(\theta+\phi)}$\\
$L_{*2} \cap L_{01}$ & $\zzz_{12}$ 
& 0 
& $e^{i\theta}\frac{\sin \phi}{\sin(\theta+\phi)}$ \\
$L_{*2} \cap L_{13}$ & $\zzz_{13}$ 
& $\frac{\sin (\beta-\theta-\phi) \sin \theta}{\sin(\beta-\theta)\sin (\theta+\phi)}$
& $e^{i\theta}\frac{\sin \phi}{\sin(\theta+\phi)}$ \\
$L_{*2} \cap L_{03}$ & $\zzz_{14}$ 
& $e^{-i\phi} \frac{\sin(\alpha -\phi) \sin (\beta-\theta-\phi) \sin \theta}
{\sin \alpha\sin(\beta-\theta)\sin (\theta+\phi)}$ 
& $e^{i\theta}\frac{\sin \phi}{\sin(\theta+\phi)}$\\
\hline
\caption{The coordinates of the vertices.}
\label{table:verticespoly}
\end{longtable}

\subsubsection{Second set of coordinates}

It will be useful to define another set of coordinates in order to define the polyhedron explicitly. 
This is in the spirit of the $\textbf w$-coordinates in the previous works and is given by 
\begin{equation}\label{eq:stcoord}
\www=
\begin{bmatrix}
\wwc 1 \\ \wwc 2 \\1
\end{bmatrix}
=P^{-1}
\begin{bmatrix}
\zzc 1 \\ \zzc 2 \\1
\end{bmatrix}.
\end{equation}

\begin{figure}[t]
\centering
\includegraphics[width=1\textwidth]{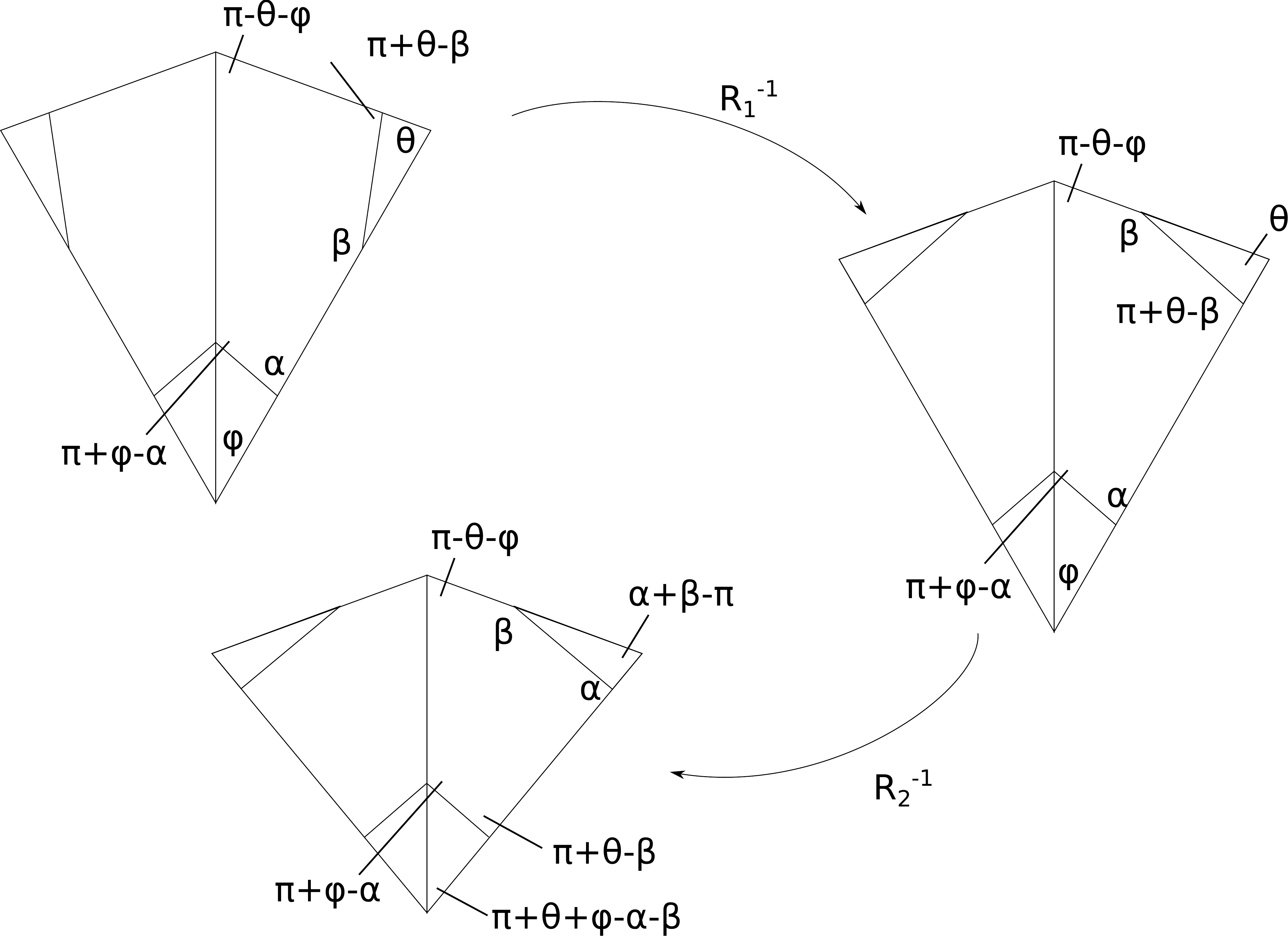}
\begin{quote}\caption{The action of $P^{-1}$ on the angles. \label{fig:P-action}} \end{quote}
\end{figure}

To calculate the $\www$-coordinates, the first thing to do is to calculate $P^{-1}(\alpha, \beta, \theta,\phi)$, with a similar argument as in Section \ref{sec:moves}.
We recall that this means that $P^{-1}$ is applied to the configuration $(\alpha, \beta, \theta,\phi)$.
As shown in Figure \ref{fig:P-action}, $P^{-1}$ acts as follows:
\begin{align}\label{eq:P-onangles}
(\alpha, \beta, \theta,\phi) &\xmapsto{\gR 1^{-1}} (\alpha', \beta', \theta',\phi')=(\alpha, \pi+\theta-\beta, \theta, \phi) \xmapsto{\gR 2^{-1}} \\
&\xmapsto{\gR 2^{-1}} (\alpha'', \beta'', \theta'',\phi'')=( \pi+\theta-\beta, \alpha, \alpha+\beta-\pi, \pi+\theta+\phi-\alpha-\beta), \nonumber
\end{align}
so 
\[
P^{-1}(\alpha, \beta, \theta,\phi)=\gR 2^{-1}(\alpha, \pi+\theta-\beta, \theta, \phi) \circ \gR 1^{-1}(\alpha, \beta, \theta,\phi).
\]

Explicitly, we have
\begin{align}\label{eq:P-matrix}
&P^{-1}(\alpha, \beta, \theta,\phi)= \nonumber \\
&=\begin{bmatrix}
-\sin \alpha \sin\theta' e^{-i(\alpha-\phi)}
&-\frac{\sin (\alpha -\phi) \sin\theta'\sin \beta}{\sin (\beta-\theta)} e^{-i(\alpha+\theta)}
&\sin (\alpha -\phi) \sin\theta' e^{-i\alpha}\\
\sin \beta \sin \phi' e^{i(\beta-\theta)}
& \sin \beta \sin \phi' e^{i(\beta-\theta)}
& -\sin \beta \sin \phi' e^{i(\beta-\theta)}  \\
\sin(\theta+\phi) \sin\alpha e^{i(\beta-\theta)}
& -\sin (\theta+\phi) \sin \beta e^{-i(\alpha+\theta)}
& A\\
\end{bmatrix},
\end{align}
where $\phi'=\pi+\theta+\phi -\alpha-\beta$, $\theta'=\alpha+\beta -\pi$ and $A$ is
\begin{align}\label{eq:Aininv}
A&=-\sin\theta'\sin \phi -\sin(\theta+\phi) \sin \alpha e^{i(\beta-\theta)} \\
&=  -\sin \phi'\sin \theta +\sin(\theta+\phi) \sin (\beta-\theta)e^{-i\alpha}.\nonumber 
\end{align}
One can easily verify that the matrices of $P(\alpha, \beta, \theta,\phi)$ and $P^{-1}(\alpha, \beta, \theta,\phi)$ in Equations \eqref{eq:Pmatrix} and \eqref{eq:P-matrix} respectively satisfy Equation \eqref{eq:inverse}.

We now apply $P^{-1}$ to the lines and vertices described in Tables \ref{table:lineseq} and \ref{table:verticespoly} to obtain their $\www$-coordinates. 

The complex lines will have $\www$-coordinates as follows.
\begin{longtable}{|c|c|}
\hline
$L_{ij}$ & Equations in terms of the $\www$-coordinates \\
\hline
$L_{*0}$ & 
$\wwc 1=-\frac{\sin(\alpha-\phi)\sin (\alpha+\beta)}{\sin (\beta-\theta) \sin (\theta+\phi)}$ \\
$L_{*1}$ & 
$\wwc 2=-e^{i(\alpha+\beta)} \frac{\sin (\alpha+\beta-\theta-\phi)}{\sin(\theta+\phi)} $  \\
$L_{*2}$ & 
$ \wwc 2=\frac{\sin(\alpha+\beta-\theta-\phi)\sin \beta}{\sin \alpha \sin(\theta+\phi)}$ \\
$L_{*3}$ & 
$\wwc 1=e^{-i(\alpha+\beta-\theta-\phi)} \frac{\sin (\alpha+\beta)}{\sin(\theta+\phi)}$\\
$L_{01}$ & 
$-\frac{\sin (\beta- \theta)}{\sin(\alpha -\phi)}e^{i(\alpha+\beta-\theta-\phi)}\wwc 1+ \wwc 2=1$ \\
$L_{02}$ & 
$-\frac{\sin (\beta- \theta)}{\sin(\alpha -\phi)}e^{i(\alpha+\beta-\theta-\phi)}\wwc 1- e^{-i(\alpha+\beta)}\frac{\sin \alpha}{\sin \beta}  \wwc 2=1$ \\
$L_{03}$ & 
$\wwc 1=0$\\
$L_{12}$ & 
$\wwc 2=0$ \\
$L_{13}$ & 
$\wwc 1+\wwc 2=1$ \\
$L_{23}$ & 
$\wwc 1- e^{-i(\alpha+\beta)}\frac{\sin \alpha}{\sin \beta} \wwc 2=1$ \\
\hline 
\caption{The equations defining the complex lines of two cone points collapsing in terms of the $\www$-coordinates.}
\label{table:lineseqwcoord}
\end{longtable}

The vertices have $\www$-coordinates as follows. 

\begin{longtable}{|c|c|c|c|}
\hline
 $\zzz_k$ & $ \wwc 1$ & $\wwc 2$ \\
\hline
$\zzz_1$ 
& $\frac{e^{-i\alpha}\sin(\alpha-\phi)\sin(\alpha+\beta) }
{ \sin (\alpha-\phi) \sin \beta-e^{-i(\theta+\phi)}\sin\alpha \sin (\beta-\theta)}$ 
& $\frac{e^{i(\beta-\theta)}\sin\beta \sin (\alpha+\beta-\theta-\phi) }
{\sin (\alpha-\phi) \sin \beta-e^{-i(\theta+\phi)}\sin\alpha \sin (\beta-\theta)} $ \\
$\zzz_2$ 
& 0 
& 0  \\
$\zzz_3$ 
& $-\frac{\sin(\alpha-\phi) \sin (\alpha+\beta)}{\sin (\beta-\theta) \sin (\theta+\phi)}$ 
& $-e^{i(\alpha+\beta)}\frac{\sin (\alpha+\theta) \sin\beta \sin (\alpha+\beta-\theta-\phi)}
{\sin (\beta-\theta) \sin \alpha \sin(\theta+\phi)}$ \\
$\zzz_4$ 
& $-\frac{\sin(\alpha-\phi) \sin (\alpha+\beta)}{\sin (\beta-\theta) \sin (\theta+\phi)}$ 
& 0 \\
$\zzz_5$ 
& $-\frac{\sin(\alpha-\phi) \sin (\alpha+\beta)}{\sin (\beta-\theta) \sin (\theta+\phi)}$ 
& $\frac{\sin (\alpha+\theta) \sin (\alpha+\beta-\theta-\phi)}{\sin (\beta-\theta) \sin(\theta+\phi)}$ \\
$\zzz_6$ 
& $\frac{\sin (\alpha+\beta) \sin (\theta+\phi-\alpha)}{\sin\beta\sin (\theta+\phi)}$
& $-e^{i(\alpha+\beta)}\frac{\sin (\alpha+\beta-\theta-\phi)}{\sin(\theta+\phi)}$  \\
$\zzz_7$ 
& $-e^{-i(\alpha+\beta-\theta-\phi)} \frac{\sin(\alpha-\phi)\sin (\alpha+\beta) \sin (\theta+\phi-\alpha)}
{\sin (\beta-\theta)\sin\beta\sin (\theta+\phi)}$ 
& $-e^{i(\alpha+\beta)}\frac{\sin (\alpha+\beta-\theta-\phi)}{\sin(\theta+\phi)}$ \\
$\zzz_8$ 
& 0 
& $-e^{i(\alpha+\beta)}\frac{\sin (\alpha+\beta-\theta-\phi)}{\sin(\theta+\phi)}$ \\
$\zzz_9$ 
& $e^{-i(\alpha+\beta-\theta-\phi)} \frac{\sin (\alpha+\beta)}{\sin(\theta+\phi)}$ 
& $\frac{\sin (\beta+\phi) \sin (\alpha+\beta-\theta-\phi)}
{\sin(\alpha-\phi) \sin (\theta+\phi)}$\\
$\zzz_{10}$ 
& $e^{-i(\alpha+\beta-\theta-\phi)} \frac{\sin (\alpha+\beta)}{\sin(\theta+\phi)}$ 
& 0\\
$\zzz_{11}$ 
& $e^{-i(\alpha+\beta-\theta-\phi)} \frac{\sin (\alpha+\beta)}{\sin(\theta+\phi)}$ 
& $-e^{i(\alpha+\beta)}\frac{\sin (\beta+\phi) \sin\beta \sin (\alpha+\beta-\theta-\phi)}
{\sin(\alpha-\phi) \sin \alpha \sin (\theta+\phi)}$\\
$\zzz_{12}$  
& $-e^{-i(\alpha+\beta-\theta-\phi)} \frac{\sin(\alpha-\phi)\sin (\theta+\phi-\beta)\sin (\alpha+\beta)}
{\sin (\beta-\theta) \sin \alpha \sin(\theta+\phi)}$ 
& $\frac{\sin \beta \sin(\alpha+\beta-\theta-\phi)}{\sin \alpha \sin(\theta+\phi)}$\\
$\zzz_{13}$ 
& $\frac{\sin (\theta+\phi-\beta) \sin (\alpha+\beta)}{\sin \alpha \sin(\theta+\phi)}$ 
& $\frac{\sin \beta \sin(\alpha+\beta-\theta-\phi)}{\sin \alpha \sin(\theta+\phi)}$ \\
$\zzz_{14}$ 
& 0 
& $\frac{\sin \beta \sin(\alpha+\beta-\theta-\phi)}{\sin \alpha \sin(\theta+\phi)}$\\
\hline
\caption{The $\textbf \www$-coordinates of the vertices.}
\label{table:verticespolywcoord}
\end{longtable}

\begin{rk}
The equations of the lines are of the same form as the ones for the $\zzz$-coordinates, except for the sign of the exponential for the $\zzc 1$-coordinate and up to substituting $(\alpha, \beta,\theta,\phi)$ with the new angles as in Figure \ref{fig:P-action}, i.e. up to substituting $(\alpha, \beta,\theta,\phi)$ with $(\alpha', \beta',\theta',\phi')=(\pi+\theta -\beta, \alpha, \alpha+\beta -\pi, \pi+\theta+\phi -\alpha-\beta)$.
The same is true for the coordinates of the vertices. 
In other words, up to remembering that $\alpha'=\pi+\theta -\beta$, $\beta'=\alpha$, $\theta'= \alpha+\beta -\pi$ and $\phi'=\pi+\theta+\phi -\alpha-\beta$ as in Figure \ref{fig:P-action}, the $\www$-coordinates can be equivalently listed as in the following tables.

\begin{longtable}{|c|c|}
\hline
$L_{ij}$ & Equations in terms of the $\www$-coordinates \\
\hline
$L_{*0}$ & $\wwc 1=\frac{\sin(\alpha'-\phi')\sin \theta'}{\sin \alpha' \sin (\theta'+\phi')}$ \\
$L_{*1}$ & $\wwc 2=e^{i\theta'} \frac{\sin \phi'}{\sin(\theta'+\phi')} $  \\
$L_{*2}$ & $ \wwc 2=\frac{\sin(\beta'-\theta')\sin \phi'}{\sin \beta' \sin(\theta'+\phi')}$ \\
$L_{*3}$ & $\wwc 1=e^{i\phi'} \frac{\sin \theta'}{\sin(\theta'+\phi')}$\\
$L_{01}$ & $\frac{\sin \alpha'}{\sin(\alpha' -\phi')}e^{-i\phi'}\wwc 1+ \wwc 2=1$ \\
$L_{02}$ & $\frac{\sin \alpha'}{\sin(\alpha' -\phi')}e^{-i\phi'}\wwc 1+ e^{-i\theta'}\frac{\sin \beta'}{\sin (\beta'-\theta')} \wwc 2=1$ \\
$L_{03}$ & $\wwc 1=0$\\
$L_{12}$ & $\wwc 2=0$ \\
$L_{13}$ & $\wwc 1+\wwc 2=1$ \\
$L_{23}$ & $\wwc 1+ e^{-i\theta'}\frac{\sin \beta'}{\sin (\beta'-\theta')} \wwc 2=1$ \\
\hline 
\caption{The equations defining the complex lines of two cone points collapsing in terms of the $\www$-coordinates and of the angles in the target configuration.}
\label{table:lineseqwcoord'}
\end{longtable}

\begin{longtable}{|c|c|c|c|}
\hline
 $\zzz_k$ & $ \wwc 1$ & $\wwc 2$ \\
\hline
$\zzz_1$ 
& $-\frac{e^{-i\beta'}\sin(\alpha'-\phi')\sin\theta' }
{\sin (\alpha'-\phi') \sin (\beta'-\theta')-e^{-i(\theta'+\phi')}\sin\alpha' \sin \beta'} $
& $-\frac{e^{-i\alpha'}\sin\phi' \sin (\beta'-\theta')}
{\sin (\alpha'-\phi') \sin (\beta'-\theta')-e^{-i(\theta'+\phi')}\sin\alpha' \sin \beta'} $ \\
$\zzz_2$ 
& 0 
& 0  \\
$\zzz_3$ 
& $\frac{\sin(\alpha'-\phi') \sin \theta'}{\sin \alpha' \sin (\theta'+\phi')}$
& $e^{i\theta'}\frac{\sin(\alpha'+\theta') \sin(\beta'-\theta') \sin \phi'}
{\sin \alpha' \sin \beta' \sin(\theta'+\phi')}$ \\
$\zzz_4$ 
& $\frac{\sin(\alpha'-\phi') \sin \theta'}{\sin \alpha' \sin (\theta'+\phi')}$
& 0 \\
$\zzz_5$ 
& $\frac{\sin(\alpha'-\phi') \sin \theta'}{\sin \alpha' \sin (\theta'+\phi')}$ 
& $\frac{\sin(\alpha'+\theta') \sin \phi'}{\sin \alpha' \sin(\theta'+\phi')}$ \\
$\zzz_6$ 
& $\frac{\sin \theta' \sin (\beta'-\theta'-\phi')}{\sin(\beta'-\theta')\sin (\theta'+\phi')}$
& $e^{i\theta'}\frac{\sin \phi'}{\sin(\theta'+\phi')}$  \\
$\zzz_7$ 
& $e^{i\phi'}\frac{\sin(\alpha' -\phi')\sin \theta' \sin (\beta'-\theta'-\phi')}
{\sin \alpha'\sin(\beta'-\theta')\sin (\theta'+\phi')}$
& $e^{i\theta'}\frac{\sin \phi'}{\sin(\theta'+\phi')}$ \\
$\zzz_8$ 
& 0 
& $e^{i\theta'}\frac{\sin \phi'}{\sin(\theta'+\phi')}$ \\
$\zzz_9$ 
& $e^{i\phi'} \frac{\sin \theta'}{\sin(\theta'+\phi')}$ 
& $\frac{\sin (\alpha'-\theta'-\phi') \sin \phi'}{\sin(\alpha' -\phi') \sin (\theta'+\phi')}$\\
$\zzz_{10}$ 
&  $e^{i\phi'} \frac{\sin \theta'}{\sin(\theta'+\phi')}$  
& 0\\
$\zzz_{11}$ 
&  $e^{i\phi'} \frac{\sin \theta'}{\sin(\theta'+\phi')}$ 
& $e^{i\theta'} \frac{\sin (\alpha'-\theta'-\phi') \sin(\beta'-\theta') \sin \phi'}
{\sin(\alpha' -\phi') \sin \beta' \sin (\theta'+\phi')}$\\
$\zzz_{12}$ 
& $e^{i\phi'}\frac{\sin(\alpha' -\phi') \sin (\beta' + \phi') \sin \theta'}
{\sin \alpha' \sin \beta' \sin(\theta'+\phi')}$ 
& $\frac{\sin (\beta' - \theta') \sin \phi'}{\sin \beta' \sin(\theta'+\phi')}$\\
$\zzz_{13}$ 
& $\frac{\sin (\beta' + \phi') \sin \theta'}{\sin \beta' \sin(\theta'+\phi')}$ 
& $\frac{\sin (\beta' - \theta') \sin \phi'}{\sin \beta' \sin(\theta'+\phi')}$ \\
$\zzz_{14}$ 
& 0 
& $\frac{\sin (\beta' - \theta') \sin \phi'}{\sin \beta' \sin(\theta'+\phi')}$\\
\hline
\caption{The $\textbf \www$-coordinates of the vertices in terms of the angles in the target configuration.}
\label{table:verticespolywcoord'}
\end{longtable}
\end{rk}

\subsubsection{The polyhedron}

Many of the vertices are contained in bisectors, and we use these bisectors to cut out a polyhedron, which will be called $D$.
In particular, and following \cite{livne}, \cite{boadiparker} and \cite{irene}, we will denote the bisectors as follows. 
\begin{center}
\begin{tabular}{|c|c|c|}
\hline
Bisector & Equation & Points in the bisector \\
\hline
$B(P)$ & $\im (\zzc 1)=0$ 
& $\zzz_1, \zzz_3,\zzz_4, \zzz_5, \zzz_{9}, \zzz_{10}, \zzz_{12}, \zzz_{13}$ \\
$B(P^{-1})$ & $\im (\wwc 1)=0 $
& $\zzz_2, \zzz_3,\zzz_4, \zzz_5, \zzz_{6}, \zzz_{8}, \zzz_{13}, \zzz_{14}$ \\
$B(J)$ & $\im (e^{i \phi} \zzc 1)=0$ 
& $\zzz_1, \zzz_6,\zzz_7, \zzz_8, \zzz_{9}, \zzz_{11}, \zzz_{12}, \zzz_{14}$ \\
$B(J^{-1})$ & $\im (e^{-i \phi} \wwc 1)=0 $
& $\zzz_2, \zzz_7,\zzz_8, \zzz_9, \zzz_{10}, \zzz_{11}, \zzz_{12}, \zzz_{14}$\\
$B(\gR 1)$ & $\im (\zzc 2)=0$ 
& $\zzz_1, \zzz_3,\zzz_4, \zzz_6, \zzz_{7}, \zzz_{9}, \zzz_{10}, \zzz_{11} $\\
$B(\gR 1^{-1})$ & $\im (e^{-i \theta} \zzc 2)=0$ 
& $\zzz_1, \zzz_3,\zzz_5, \zzz_6, \zzz_{8}, \zzz_{12}, \zzz_{13}, \zzz_{14}$\\
$B(\gR 2)$ & $\im (\wwc 2)=0 $
& $\zzz_2, \zzz_4,\zzz_5, \zzz_9, \zzz_{10}, \zzz_{12}, \zzz_{13}, \zzz_{14}$\\
$B(\gR 2^{-1})$ & $\im (e^{-i \theta} \wwc 2)=0 $
& $\zzz_2, \zzz_3,\zzz_4, \zzz_6, \zzz_{7}, \zzz_{8}, \zzz_{10}, \zzz_{11}$\\
\hline
\end{tabular}
\end{center}

Bisectors are a special type of hypersurfaces in complex hyperbolic space. 
More details about bisectors can be found in \cite{goldman}.
Bisectors are defined as the locus of points which are equidistant from two given points $\textbf{z}_i$ and $\textbf{z}_j$.
The complex line $L$ spanned by $\textbf z_i$ and $\textbf z_j$ is called the \emph{complex spine} of the bisector. 
The intersection between the complex spine and the bisector is a geodesic $\gamma \in L$, which is called the \emph{spine} of the bisector. 

Bisectors are foliated by totally geodesic subspaces in two different ways.
The first foliation is by slices. 
A \emph{slice} is a complex line that is a fibre of the map $\Pi_L$, the orthogonal projection to the complex spine $L$.
The bisector is the preimage by $\Pi_L$ of $\gamma$ and the preimage of each point of $\gamma$ is a slice. 
The second foliation is by meridians.
A \emph{meridian} is a totally geodesic Lagrangian plane containing the spine $\gamma$. 
It can also be described as the set of fixed points of an antiholomorphic involution which swaps $\textbf z_i$ and $\textbf z_j$.

Bisectors also contain Giraud discs.
For three points $\textbf{z}_i,\textbf{z}_j$ and $\textbf{z}_k$, not all contained in the same complex line, one can consider $B(\textbf z_i,\textbf z_j, \textbf z_k)$, the set of points equidistant from these three points. 
Such set is a smooth non totally geodesic disc, contained in exactly three bisectors ($B(\textbf z_i,\textbf z_j), B(\textbf z_i,\textbf z_k)$ and $B(\textbf z_j,\textbf z_k)$) and it's called a \emph{Giraud disc}.

The reason for the bisectors to be denoted as $B(T)$ is that we want the map $T$ to send the side $B(T)$ to $B(T^{-1})$. 
The following lemma shows that this is the case.

\begin{lemma}\label{lemma:bisectors}
In \emph{$\zzz$}- and \emph{$\www$}-coordinates and writing $\theta'= \alpha+\beta -\pi$ and $\phi'=\pi+\theta+\phi -\alpha-\beta$, we have 
\begin{itemize}
\item $\im(\zzc 1)\leq0$ if and only if \emph{
$\lvert \langle \zzz, \textbf n_{*1} \rangle \rvert 
\leq \lvert \langle \zzz, P^{-1}(\textbf n_{*3}) \rangle \rvert$},
\item $\im(\wwc 1)\geq0$ if and only if \emph{
$\lvert \langle \www, \textbf n_{*3} \rangle \rvert 
\leq \lvert \langle \www, P(\textbf n_{*1}) \rangle \rvert$},
\item $\im(e^{i \phi} \zzc 1)\geq0$ if and only if \emph{
$\lvert \langle \zzz, \textbf n_{*0} \rangle \rvert 
\leq \lvert \langle \zzz, J^{-1}(\textbf n_{*0}) \rangle \rvert$},
\item $\im(e^{-i \phi'} \wwc 1)\leq0$ if and only if \emph{
$\lvert \langle \www, \textbf n_{*0} \rangle \rvert 
\leq \lvert \langle \www, J(\textbf n_{*0}) \rangle \rvert$},
\item $\im(\zzc 2)\geq0$ if and only if \emph{
$\lvert \langle \zzz, \textbf n_{*2} \rangle \rvert 
\leq \lvert \langle \zzz, \gR 1^{-1}(\textbf n_{*3}) \rangle \rvert$},
\item $\im(e^{-i \theta} \zzc 2)\leq0$ if and only if \emph{
$\lvert \langle \zzz, \textbf n_{*3} \rangle \rvert 
\leq \lvert \langle \zzz, \gR 1(\textbf n_{*2}) \rangle \rvert$},
\item $\im(\wwc 2)\geq0$ if and only if \emph{
$\lvert \langle \www, \textbf n_{*1} \rangle \rvert 
\leq \lvert \langle \www, \gR 2^{-1}(\textbf n_{*2}) \rangle \rvert$},
\item $\im(e^{-i \theta'} \wwc 2)\leq0$ if and only if \emph{
$\lvert \langle \www, \textbf n_{*2} \rangle \rvert 
\leq \lvert \langle \www, \gR 2(\textbf n_{*1}) \rangle \rvert$}.
\end{itemize}
\end{lemma}

The proof goes like the one of equivalent lemmas in the previous works. 
In particular, it can be found in Lemma 4.6 in \cite{livne}, Lemma 4.2 in \cite{boadiparker} and in Lemma 7.2 of \cite{irene}.
One just needs to remark that $\textbf n_{*i}$ depends on the configuration we are using. 
So, for example, the first line of the lemma is
\[
\lvert \langle \zzz, \textbf n_{*1}(\alpha, \beta, \theta, \phi) \rangle \rvert 
\leq \lvert \langle \zzz, P^{-1}(\textbf n_{*3}(\beta, \pi+\theta-\beta, \theta+\alpha-\beta, \phi+\beta-\alpha)) \rangle \rvert,
\]
since $(\alpha, \beta, \theta, \phi) \xmapsto{P} (\beta, \pi+\theta-\beta, \theta+\alpha-\beta, \phi+\beta-\alpha)$.
The rest is similar.

Now the polyhedron $D=D(\alpha,\beta,\theta,\phi)$ is defined as the intersection of all the half spaces in the lemma. 
More precisely, it will be
\begin{equation} \label{eq:polygeneric}
D(\alpha, \beta,\theta,\phi)= \left\{ \zzz=P(\www) \colon
\begin{array}{l l}
\arg(\zzc 1) \in (-\phi,0), & \arg(\zzc 2) \in (0,\theta),\\
\arg(\wwc 1) \in (0,\phi'), & \arg(\wwc 2) \in (0,\theta')
\end{array} \right\},
\end{equation}
where, as before, we have $\theta'= \alpha+\beta -\pi$ and $\phi'=\pi+\theta+\phi -\alpha-\beta$.

The sides (codimension 1 cells) of the polyhedron will be defined as $S(T)=D \cap B(T)$. 
Each of them is contained in one of the bisectors in the table. 


\subsection{The combinatorial structure of the polyhedron}\label{sec:degen}

We now want to study the combinatorics of the polyhedron $D(\alpha, \beta, \theta, \phi)$.

First, following \cite{irene}, we will see how the combinatorics change with the values of the angles. 
Later we will study all possible side (3-dimensional facets) intersections in order to be able to list all possible ridges (2-dimensional facets) and edges (1-dimensional facets).

According to the values of the parameters, we will have occasions where the the three vertices on $L_{*i}$ collapse to a single vertex, for $i=0,1,2,3$. 

\begin{prop}\label{prop:gencollapse}
We have
\begin{itemize}
\item the vertices on $L_{*0}$ collapse when $\alpha-\phi \geq \pi-\theta-\phi$, i.e. when $\pi-\alpha-\theta \leq 0$;
\item the vertices on $L_{*1}$ collapse when $\pi-\alpha \geq \pi-\theta-\phi$, i.e. $\alpha-\theta-\phi \leq 0$;
\item the vertices on $L_{*2}$ collapse when $\sin(\beta-\theta)/\sin \beta \leq \sin \phi / \sin(\theta+\phi)$, i.e. $\beta-\theta-\phi \leq 0$;
\item the vertices on $L_{*3}$ collapse when $\pi-\beta \leq \phi$, i.e. $\pi-\beta-\phi \leq 0$. 
\end{itemize}
\end{prop}

In fact, for example, the vertices $\zzz_3$, $\zzz_4$ and $\zzz_5$ on $L_{*0}$ collapse if, when making $T_1$ as big as possible, before we can have $v_0 \equiv v_*$, we have that $v_1$ hits the left-hand vertex of $T_2$ and so $v_1 \equiv v_2 \equiv v_3$. 
This implies that there is no other choice for $z_2$ but to be zero, instead of having the three choices that give the three possible vertices having $v_0 \equiv v_*$.
Translated on the parameters, this gives that $\alpha-\phi \geq \pi-\theta-\phi$.
The others can be verified in a similar way.

We remark that the sides all have the same combinatorial structure. 
In particular, they will look like in Figure \ref{fig:sidecombinatorics}.
This is the same structure as the one of the sides of the polyhedron in \cite{irene} and first appeared as the combinatorial structure of 2 of the 10 sides in \cite{type2}. 
Each side will correspond to fixing the argument of one of the coordinates. 
Then there will be one triangular ridge (e.g. the bottom one) where the coordinate is equal to zero and a second triangular ridge (e.g. the top one) where the coordinate has another fixed value. 
The complex lines interpolating between the two will be the slices of the foliation mentioned earlier in this section.
The edge connecting the two triangles is contained in the complex spine of the bisector and always contains one of the vertices $\zzz _1$ or $\zzz _2$. 
The pentagonal side ridges containing the vertical edge are contained in totally geodesic Lagrangian planes and are the extremities of the foliation by meridians. 
We claim that in each side the modulus of the coordinate we are considering varies between the two values it assumes on the top and bottom triangular ridges. 
To check this, for example, in $S(J)$, we need to check that $| \zzc 1 |$ in $\zzz_{11}$ and $\zzz_{14}$ is smaller than $|t_1|$ in $\zzz_6$, $\zzz_7$ and $\zzz_8 $ ($|t_1|$ has the same value in these three vertices, since they are contained in the complex line $L_{*1}$) and so on. 
It is easy to check that this is true for each side as long as 
\begin{align}\label{eq:PP-possible}
\sin (\alpha+\beta-\pi) &\geq 0 
& \sin(\pi+\alpha+\beta-\theta-\phi) &\geq 0 \nonumber\\
 \sin(\alpha+\theta-\beta) &\geq 0
& \sin(\beta+\phi-\alpha) &\geq 0.
\end{align}
Remembering the action of $P$ and $P^{-1}$ on the angles, this means that we are just asking for the configuration after applying these two maps to make sense in our coordinates. 

\begin{figure}
\centering
\includegraphics[width=0.5\textwidth]{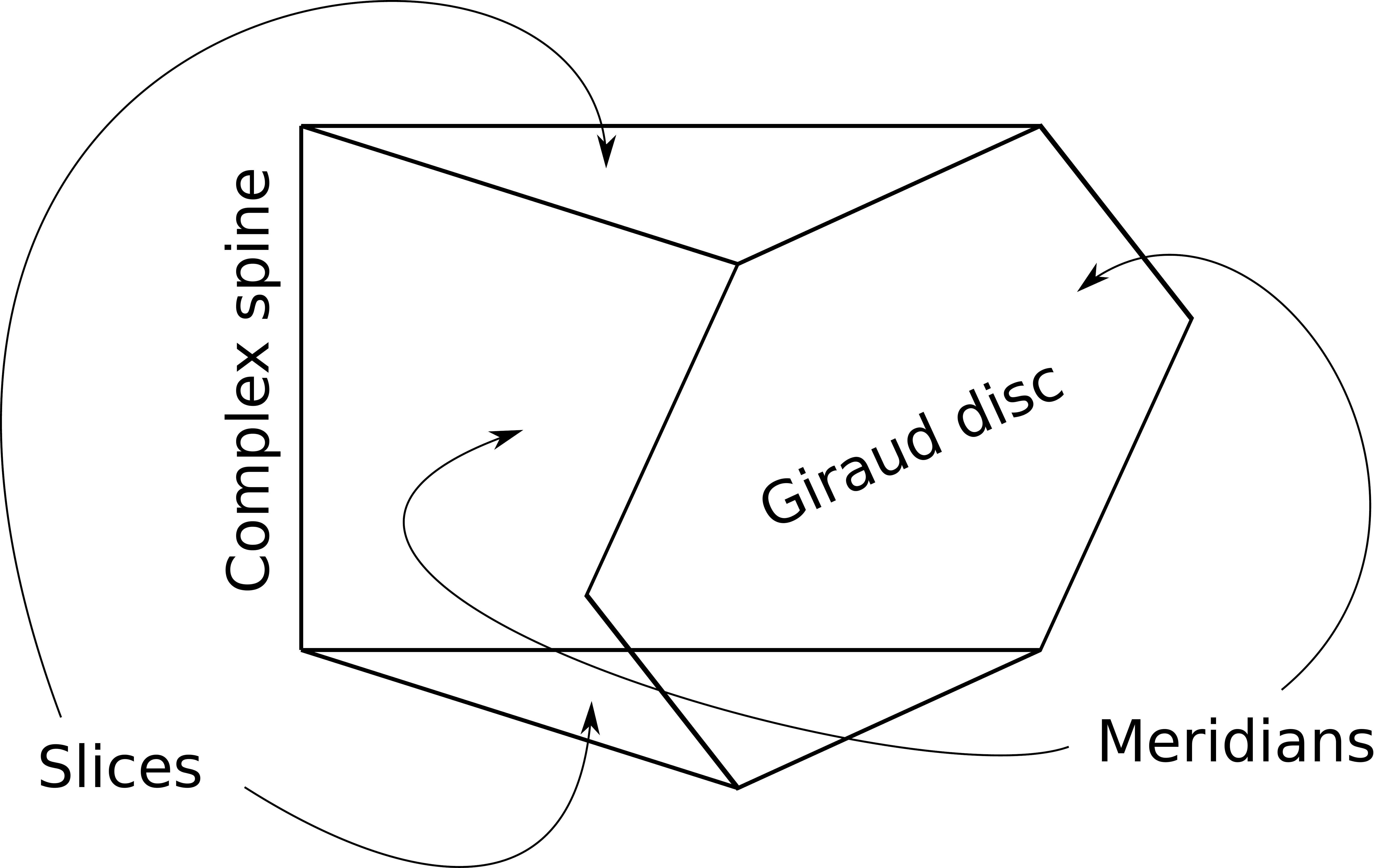}
\begin{quote}\caption{The combinatorial structure of a side. \label{fig:sidecombinatorics}} \end{quote}
\end{figure}

This gives us the following lemma:
\begin{lemma}\label{lemma:sidecombinatorics}
If the parameters satisfy \eqref{eq:PP-possible}, then 
\begin{itemize}
\item In $SP)$, we have 
$| \zzc 1| \leq \frac{\sin(\alpha-\phi) \sin\theta}{\sin \alpha \sin(\theta+\phi)}$,
\item In $S(J)$, we have 
$| \zzc 1| \leq \frac{\sin\theta}{\sin(\theta+\phi)}$,
\item In $S(\gR 1)$, we have 
$| \zzc 2| \leq \frac{\sin(\beta-\theta)\sin \phi}{\sin\beta \sin(\theta+\phi)}$,
\item In $S(\gR 1^{-1})$, we have 
$| \zzc 2| \leq \frac{\sin\phi}{\sin(\theta+\phi)}$,
\item In $S(P^{-1})$, we have 
$| \wwc 1| \leq -\frac{\sin(\alpha-\phi) \sin(\alpha+\beta)}{\sin(\beta-\theta) \sin(\theta+\phi)}$,
\item In $S(J^{-1})$, we have 
$| \wwc 1| \leq -\frac{\sin(\alpha+\beta)}{\sin(\theta+\phi)}$,
\item In $S(\gR 2)$, we have 
$| \wwc 2| \leq \frac{\sin(\alpha+\beta-\theta-\phi) \sin \beta}{\sin \alpha \sin(\theta+\phi)}$,
\item In $S(\gR 2^{-1})$, we have 
$| \wwc 2| \leq \frac{\sin(\alpha+\beta-\theta-\phi) \sin \beta}{\sin \alpha \sin(\theta+\phi)}$,
\end{itemize}
\end{lemma}

We now want to consider all possible side intersections to find the combinatorics of the polyhedron. 
We will denote by $\gamma_{i,j}$ the geodesic segment between the vertices $\zzz _i$ and $\zzz _j$.

\begin{prop}
The following side intersections consist of the union of two edges: 
\begin{align*}
S(P) \cap S(J^{-1})&= \gamma_{10,9} \cup \gamma_{9,12}, &
S(\gR 1^{-1}) \cap S(J^{-1}) &= \gamma_{8,14} \cup \gamma_{14,12}, \\
S(P) \cap S(\gR 2^{-1})&= \gamma_{3,4} \cup \gamma_{4,10}, &
S(J) \cap S(\gR 2)&= \gamma_{9,12} \cup \gamma_{12,14}, \\
S(\gR 1) \cap S(\gR 2) &= \gamma_{4,10} \cup \gamma_{10,9}, &
S(J) \cap S(P^{-1}) &= \gamma_{6,8} \cup \gamma_{8,14}, \\
S(\gR 1) \cap S(P^{-1}) &= \gamma_{4,3} \cup \gamma_{3,6}, &
S(\gR 1^{-1}) \cap S(\gR 2^{-1}) &= \gamma_{3,6} \cup \gamma_{6,8}.
\end{align*}
\end{prop}

The proof of this proposition follows exactly the one in Appendix A of \cite{livne} and Proposition 7.8 in \cite{irene}.

\begin{prop}
The bisectors satisfy:
\begin{itemize}
\item A point \emph{$\zzz$} in the side intersection $S(P) \cap S(P^{-1})$, with $\zzc 1 \neq \frac{\sin \theta \sin (\alpha - \phi)}{\sin \alpha \sin (\theta+\phi)}, \wwc 1 \neq -\frac{\sin (\alpha+\beta) \sin (\alpha - \phi)}{\sin(\beta-\theta)\sin (\theta+\phi)}$, belongs to the edge $\gamma_{5,13}$.
\item A point \emph{$\zzz$} in the side intersection $S(J) \cap S(\gR 2^{-1})$, with $\zzc 1 \neq e^{-i \phi} \frac{\sin \theta}{\sin (\theta+\phi)}$ and $ \wwc 2 \neq -e^{i (\alpha+\beta)} \frac{\sin (\alpha+\beta-\theta-\phi)}{\sin (\theta+\phi)}$, belongs to the edge $\gamma_{7,11}$.
\item Moreover, a point \emph{$\zzz$} in the side intersection $S(\gR 2) \cap S(\gR 1^{-1})$, with $\zzc 2 \neq e^{i \theta}\frac{\sin \phi}{\sin (\theta+\phi)}$ and $\wwc 2 \neq \frac{\sin (\alpha+\beta-\theta-\phi) \sin \beta}{\sin \alpha\sin (\theta+\phi)}$, belongs to the edge $\gamma_{5,13}$.
\item Finally, a point \emph{$\zzz$} in the side intersection $S(\gR 1) \cap S(J^{-1})$, with $\zzc 2 \neq \frac{\sin(\beta-\theta)\sin \phi}{\sin \beta\sin (\theta+\phi)}$ and $\wwc 1 \neq e^{-i (\alpha+\beta-\theta-\phi)} \frac{\sin (\alpha+\beta)}{\sin (\theta+\phi)}$, belongs to the edge $\gamma_{7,11}$.
\end{itemize}
\end{prop}

We will prove the first point and the others are proved in the exact same way.
The proof is very similar to the ones in \cite{livne} and \cite{irene}.

\begin{proof}
Let us take $\zzz \in S(P) \cap S(P^{-1})$. 
Then 
\begin{align*}
\zzc 1 &=x, & \wwc 1 &=u
\end{align*}
and by hypothesis and using Lemma \ref{lemma:sidecombinatorics} we have
\begin{align}\label{eq:ineqinproof}
x &\leq \frac{\sin \theta \sin (\alpha - \phi)}{\sin \alpha \sin (\theta+\phi)}, 
& u &\leq -\frac{\sin (\alpha+\beta) \sin (\alpha - \phi)}{\sin(\beta-\theta)\sin (\theta+\phi)}.
\end{align}

Then using \eqref{eq:stcoord} one can express $\zzc 2$ and $\wwc 2$ in terms of $x$ and $u$ as follows:

\begin{multline*} 
(\sin(\theta+\phi)\sin \alpha x -\sin(\alpha-\phi)\sin \theta) 
\wwc 2= \\
-\sin(\beta-\theta) \sin \theta e^{i(\alpha+\beta-\theta-\phi)} u
+\sin(\theta+\phi) \sin(\beta-\theta) e^{i(\alpha+\beta-\theta)} ux\\
+(\sin(\alpha+\beta) \sin \phi e^{i(\beta -\theta)} 
+\sin(\theta+\phi)\sin \alpha )x 
-\sin(\alpha-\phi) \sin \theta 
\end{multline*}

\begin{multline*}
\left(-\sin(\theta+\phi)\sin(\beta-\theta)u-\sin(\alpha-\phi)\sin(\alpha+\beta)\right)
\zzc 2 = \\
\frac
{\sin(\beta-\theta)}{\sin \beta}e^{-i \theta} 
(\sin \theta \sin(\alpha+\beta-\theta-\phi) e^{i \alpha}u
-\sin(\theta+\phi) \sin \alpha e^{i(\alpha+\beta-\theta)}ux \\
+(\sin\alpha \sin(\alpha+\beta) e^{-i \phi}
-\sin(\alpha-\phi) \sin(\alpha+\beta))x
-\sin(\theta+\phi)\sin(\beta-\theta) ).
\end{multline*}

Now, we know by Lemma \ref{lemma:bisectors} that inside $D$ we have 
\begin{multline*}
0\geq \im e^{-i\theta}\zzc 2 = 
\frac{\sin(\beta-\theta) \sin \alpha}{\sin \beta} \cdot\\
\cdot \frac{\sin(\alpha+\beta)\sin \phi x
+\sin(\theta+\phi) \sin(\alpha+\beta-\theta) ux 
-\sin\theta \sin(\alpha+\beta-\theta-\phi) u}
{\sin(\theta+\phi)\sin(\beta-\theta) u +\sin(\alpha-\phi) \sin(\alpha+\beta)},
\end{multline*}
but by \eqref{eq:ineqinproof} we know that the denominator is strictly negative and so the numerator must be positive. 

Again by Lemma \ref{lemma:bisectors}, $\zzz$ satisfies 
\begin{multline*}
0\leq \im \wwc 2 = \\
\sin(\beta-\theta)\cdot \frac{\sin(\alpha+\beta)\sin \phi x
+\sin(\theta+\phi) \sin(\alpha+\beta-\theta) ux 
-\sin\theta \sin(\alpha+\beta-\theta-\phi) u}
{\sin(\theta+\phi)\sin \alpha x -\sin(\alpha-\phi)\sin \theta},
\end{multline*}
and since by \eqref{eq:ineqinproof} the denominator must be strictly negative, then the numerator must be negative. 

But since the two numerators coincide, then they must be both equal 0. 
This means that the point we are considering must be also in $S(\gR 1^{-1})$ and in $S(\gR 2)$, which means that we are on edge $\gamma_{5,13}$. 
\end{proof}

\begin{rk}\label{rk:kcomb}
The proof relies on Lemma \ref{lemma:sidecombinatorics}. 
As we will see in Section \ref{sec:kneg}, there are cases in which \eqref{eq:PP-possible} is not satisfied. 
In term of configurations, this means that one needs to consider a slightly different configuration of triangles (see Section \ref{sec:kneg}). 
Using the new configuration one can prove an equivalent statement using the exact same strategy of proof as in \cite{livne} and \cite{irene}.
\end{rk}

\section{The polyhedron in our case}\label{sec:2f}

We will now consider the case where two of the cone points have same cone angles. 
First we will describe which sets of cone angles give a lattice, then we will show how to use the polyhedron in Section \ref{sec:config} to build a fundamental domain for them.

\subsection{Lattices with 2-fold symmetry} \label{sec:values}

As mentioned in the introduction (Section \ref{sec:intro}), the lattices we are considering were introduced by Deligne and Mostow starting from a ball 5-tuple. 
This is equivalent to consider a cone metric on a sphere of area 1 with prescribed cone singularities of angles $(\theta_0, \dots, \theta_4)$, with $0<\theta_i<2\pi$ and satisfying the discrete Gauss-Bonnet formula.
A sphere with $5$ cone points has a structure of a two dimensional complex hyperbolic space, as proved by Thurston in \cite{thurston} and showed in Section \ref{sec:config}.

Between these, we will consider the lattices with 2-fold symmetry, which means that two of the five cone points will have same con angle. 
We will assume that the 2-fold symmetry is given by $\theta_1=\theta_2$.
Occasionally the lattices will have an extra symmetry and we will also have $\theta_0=\theta_3$.
We will use the parameters in \eqref{eq:defparam} to describe the lattices, except that now $\alpha=\beta$.

By similarity with the 3-fold symmetry case, to each lattice we will associate numbers $p,p',k,k',l,l',d$, which are the orders of some maps in the group and are defined as follows.
\begin{align}\label{eq:p,p',k,k',l,l',d}
\frac{\pi}{p}&=\theta, & \frac{\pi}{k}&=\phi, & \frac{\pi}{l}&=\alpha-\theta-\phi, & \frac{\pi}{d}=\pi-\alpha-\theta \\
\frac{\pi}{p'}&=\alpha-\frac{\pi}{2}, & \frac{\pi}{k'}&=\pi+\theta+\phi-2\alpha, & \frac{\pi}{l'}&=\pi-\alpha-\phi. \nonumber
\end{align}
In particular, we will use $(p,k,p')$ to denote the configuration $(\alpha, \theta,\phi)$ and give the other values in terms of these.
Remark that in the 2-2-fold symmetry case (i.e. when we also have $\theta_0=\theta_3$), we have $\theta=\phi$ and so the lattice is of the form $(p,p,p')$.
Notice also that in the 3-fold symmetry case one would have $k=k'$, $l=l'$ and $p=2p'$. 
In fact $k$ and $k'$ will be the orders of $\gA 1(\alpha, \beta, \theta, \phi)$ for two of the different configurations we will consider (see \eqref{eq:A1matrix}) which coincide in the 3-fold symmetry case.
A similar thing happens for $l$ and $l'$. 
The values $p'$ and $p$ here are the orders of $\gR 1 \ctwo$ and $\gR 1 \circ \gR 1 \cthree$ respectively (remember that the composition is done as in \eqref{eq:composition}), and notice that they are applied to different configurations.
Since in the 3-fold symmetry case the three configurations we consider coincide, $p$ will be the order of the square of $\gR 1$, which has order $p'$ and hence $p=2p'$. 

In the following table we give the values of the cone angles for the lattice $(p,k,p')$.
These are all the lattices with 2-fold symmetry in the original list by Deligne and Mostow and together with the 3-fold symmetry lattices form the whole list of Deligne-Mostow lattices in dimension 2.
\begin{longtable}{|c||c|c|c|c|c|}
\hline
Lattice & $\theta_0$ & $\theta_1$ & $\theta_2$ & $\theta_3$ & $\theta_4$\\
\hline
(6,6,3) & $2\pi/3$ & $5\pi/3$ & $5\pi/3$ & $2\pi/3$ & $4\pi/3$ \\
(10,10,5) & $4\pi/5$ & $7\pi/5$ & $7\pi/5$ & $4\pi/5$ & $8\pi/5$ \\
(12,12,6) & $5\pi/6$ & $4\pi/3$ & $4\pi/3$ & $5\pi/6$ & $5\pi/3$ \\
(18,18,9) & $8\pi/9$ & $11\pi/9$ & $11\pi/9$ & $8\pi/9$ & $16\pi/9$ \\
(4,4,3) & $5\pi/6$ & $5\pi/3$ & $5\pi/3$ & $5\pi/6$ & $\pi$ \\
(4,4,5) & $11\pi/10$ & $7\pi/5$ & $7\pi/5$ & $11\pi/10$ & $\pi$  \\
(4,4,6) & $7\pi/6$ & $4\pi/3$ & $4\pi/3$ & $7\pi/6$ & $\pi$ \\
(3,3,4) & $7\pi/6$ & $3\pi/2$ & $3\pi/2$ & $7\pi/6$ & $2\pi/3$ \\
(3,3,3) & $\pi$ & $5\pi/3$ & $5\pi/3$ & $\pi$ & $2\pi/3$ \\
\hline
(2,6,6) & $\pi$ & $4\pi/3$ & $4\pi/3$ & $5\pi/3$ & $2\pi/3$ \\
(2,4,3) & $5\pi/6$ & $5\pi/3$ & $5\pi/3$ & $4\pi/3$ & $\pi/2$ \\
(2,3,3) & $\pi$ & $5\pi/3$ & $5\pi/3$ & $4\pi/3$ & $\pi/3$ \\
(3,4,4) & $\pi$ & $3\pi/2$ & $3\pi/2$ & $7\pi/6$ & $5\pi/6$ \\
\hline
\caption{The lattices we are considering.}
\label{table:values}
\end{longtable}

The following table records the values of the parameters in \eqref{eq:p,p',k,k',l,l',d} for each of the lattices we are considering.  

\begin{longtable}{|c|c|c|c|c|c|c|c|}
\hline
Lattice & $p$ & $k$ & $p'$ & $k'$ & $l$ & $l'$ & $d$ \\
\hline
(6,6,3) & 6 & 6 & 3 & -3 & 2 & $\infty$ & $\infty$ \\
(10,10,5) & 10 & 10 & 5 & -5 & 2 & 5 & 5 \\
(12,12,6) & 12 & 12 & 6 & -6 & 2 & 4 & 4 \\
(18,18,9) & 18 & 18 & 9 & -9 & 2 & 3 & 3 \\
(4,4,3) & 4 & 4 & 3 & -6 & 3 & -12 & -12 \\
(4,4,5) & 4 & 4 & 5 & 10 & 5 & 20 & 20 \\
(4,4,6) & 4 & 4 & 6 & 6 & 6 & 12 & 12 \\
(3,3,4) & 3 & 3 & 4 & 6 & 12 & -12 & -12 \\
(3,3,3) & 3 & 3 & 3 & $\infty$ & 6 & -6 & -6 \\
\hline
(2,6,6) & 2 & 6 & 6 & 3 &$\infty$ & 6 & -6 \\
(2,4,3) & 2 & 4 & 3 & 12 & 12 & -12 & -3 \\
(2,3,3) & 2 & 3 & 3 & 6 & $\infty$ & -6 & -3 \\
(3,4,4) & 3 & 4 & 4 & 12 & 6 & $\infty$ & -12 \\
\hline
\caption{The values of the parameters for our lattices.}
\label{table:valuespar}
\end{longtable}
%

\subsection{The fundamental polyhedron}\label{sec:2fold}

\subsubsection{Definition}

In this section we will see how one can use the general polyhedron described in Section \ref{sec:genpoly} to build a fundamental domain for Deligne-Mostow lattices with 2-fold symmetry. 
From now on we will consider a sphere with cone singularities $(\theta_0, \theta_1, \theta_2,\theta_3, \theta_4)$ in the list in Section \ref{sec:values}.
This means that we have two equal angles at the vertices $v_1$ and $v_2$. 
In the configurations as described in Section \ref{sec:config}, this means that $\alpha=\beta$.
Since the case that we treated before is when the three angles at $v_1$, $v_2$ and $v_3$ were equal, by analogy we also want to consider the configurations where the two equal angles are at $v_2$ and $v_3$ or at $v_1$ and $v_3$.
We will call these configurations of type \conf 1, \conf 2 and \conf 3 respectively.
Remark that configuration of type \conf i corresponds to having the cone angles satisfying $\theta_i=\theta_{i+1}$, for indices $i=1,2,3$ taken mod 3. 

We will build a polyhedron for each of these cases and use their union to build a fundamental domain for the lattices. 
On the parameters $(\alpha, \beta, \theta,\phi)$, type \conf 1 corresponds to $(\alpha, \alpha, \theta, \phi)$, type \conf 2 corresponds to $(\pi+\theta-\alpha, \alpha, 2\alpha-\pi, \pi+\theta+\phi-2\alpha)$ and type \conf 3 corresponds to $(\alpha, \pi+\theta-\alpha, \theta, \phi)$. 
For each type, we will consider the $\zzz$-coordinates and $\www$-coordinates.
We will have $\textbf{x}$-, $\textbf{y}$- and $\textbf{z}$-coordinates as $\zzz$-coordinates of the configuration of type \conf 1, \conf 2 and \conf 3 respectively. 
We will also have $\textbf{u}$-, $\textbf{v}$- and $\textbf{w}$-coordinates, representing copies of type \conf 1, \conf 2 and \conf 3 respectively and being the $\www$-coordinates of one of the previous ones.
More precisely, the relation between $\textbf{x}$-, $\textbf{y}$-, $\textbf{z}$- and $\textbf{u}$-, $\textbf{v}$-, $\textbf{w}$-coordinates is as follows. 
Since $P^{-1}$ acts on the copies as explained in Figure \ref{fig:P-action}, then, for example, a configuration of type \conf 1 will be sent to one of type \conf 2. 
This means that the coordinates defined as $P^{-1}(\textbf{x})$ will be the $\textbf{v}$-coordinates. 
With a similar argument, one gets
\begin{align} \label{eq:coorduvw}
\textbf u &=P^{-1}(\textbf{z}) &
\textbf v &=P^{-1}(\textbf{x}) &
\textbf w &=P^{-1}(\textbf{y}) &
\end{align}
In other words, the $\textbf{u}$-, $\textbf{v}$- and $\textbf{w}$-coordinates will be the coordinates for the configuration of type \conf 1, \conf 2 and \conf 3 respectively, obtained after applying $P$ to the standard configuration of type \conf 3, \conf 1 and \conf 2 respectively. 

We will start from the configuration of type \conf 3, with its $\textbf z$-coordinates as the $\zzz$-coordinates of configuration $(\alpha, \pi+\theta -\alpha,\theta,\phi)$.
The $\textbf x$- and $\textbf y$-coordinates will be determined by the action of the moves $\gR 1$ and $\gR 2^{-1}$ respectively. 
See Figure \ref{fig:stdconf} for more details.
As mentioned, each configuration will give us a polyhedron of the same type as $D$ in \eqref{eq:polygeneric}.

We will first explain what is the relation between the $\textbf{x}$-, $\textbf{y}$- and $\textbf{z}$-coordinates. 
Since copies of type \conf 1 and \conf 3 are swapped by $\gR 1$, it is natural to define 
\begin{equation}\label{eq:relationcoordxz}
\textbf x= \gR 1 (\alpha, \alpha, \theta,\phi) \textbf z.
\end{equation}

Since the $\textbf w$- and $\textbf u$-coordinates are also of type \conf 3 and \conf 1 respectively, one would also want 
\begin{equation}\label{eq:relationcoorduw}
\textbf u= \gR 1 \cthree \textbf w.
\end{equation}
Using the definition of $\textbf u$- and $\textbf w$-coordinates, together with the previous formula, the $\textbf{y}$-coordinates are defined as 
\begin{equation} \label{eq:relationcoordyz}
\textbf z= \gR 2 (\pi+\theta-\alpha, \alpha, 2\alpha-\pi, \pi+\theta+\phi-2\alpha) \textbf y.
\end{equation}
Using Equations \eqref{eq:coorduvw}, \eqref{eq:relationcoordxz} and \eqref{eq:relationcoordyz}, one can also see that 
\begin{equation}\label{eq:relationcoordyv}
\textbf{v}=P^{-1} \textbf{x}=P^{-1}\gR 1 \textbf{z}=P^{-1}\gR 1\gR 2 \textbf{y}= \textbf{y.}
\end{equation}

The following digram summarises the relations on the coordinates. 
\[
\begin{tikzcd}[column sep=huge, row sep=huge]
\textbf y \arrow[drr,"=", near start,swap] \arrow{r}{\gR 2} 
& \textbf z \arrow{r}{\gR 1} 
&\textbf x \\
\textbf w \arrow{u}{P} \arrow{r}[swap]{\gR 1}  
& \textbf u \arrow[u, crossing over, "P", near start] \arrow[r,"P^{-1} \gR 1 P", swap]
& \textbf v \arrow{u}{P} 
\end{tikzcd}
\]

\begin{figure}
\centering
\includegraphics[width=1\textwidth]{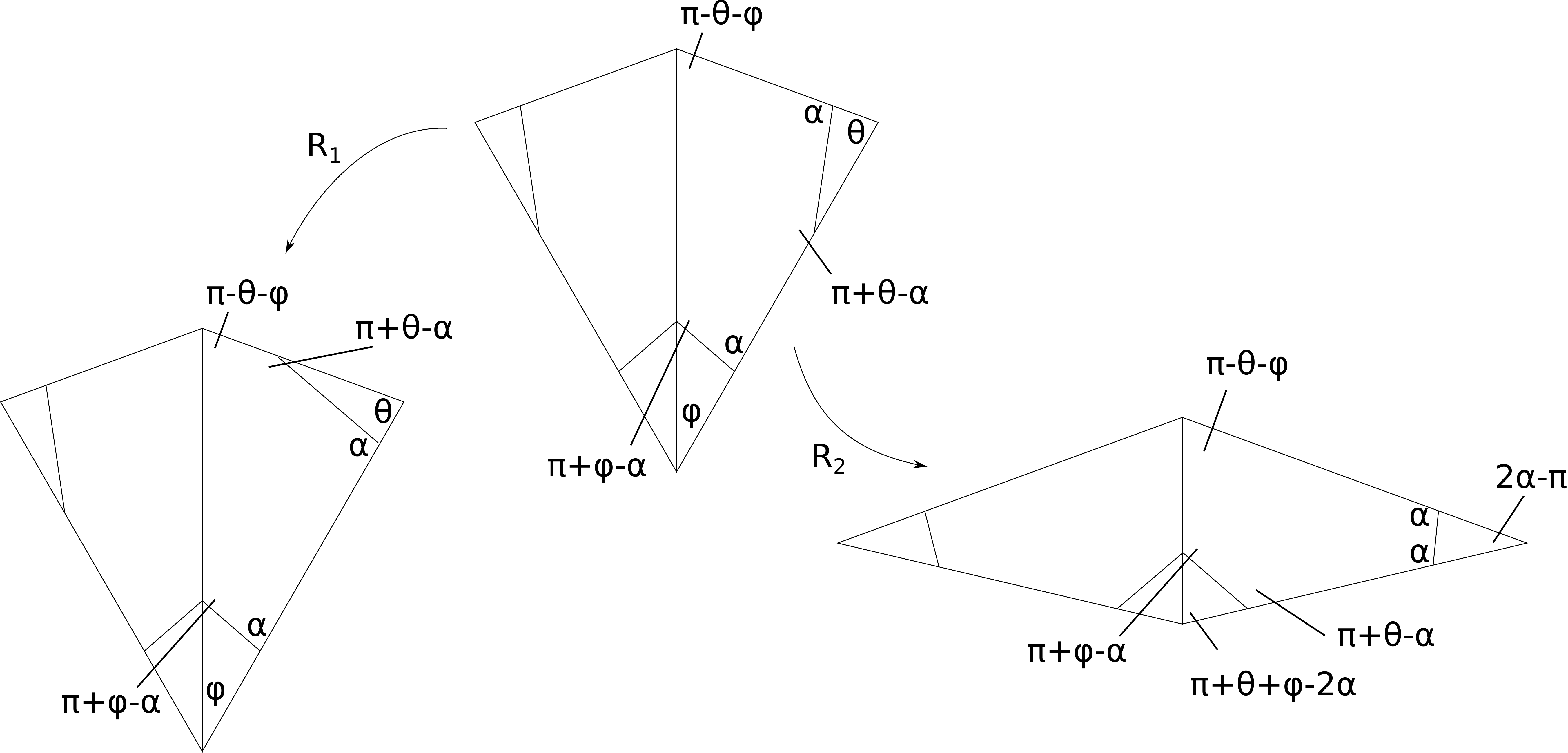}
\begin{quote}\caption{The representative for each configuration type. \label{fig:stdconf}} \end{quote}
\end{figure}

For each coordinate type, we can define a polyhedron as in \eqref{eq:polygeneric}. 
This will give us three components of our fundamental polyhedron $D$ and we will write
\begin{gather} \label{eq:Djohn}
D= D_1 \cup D_2 \cup D_3,
\textit{ with }
\begin{cases}
D_1=D(\alpha, \alpha, \theta,\phi)=\gR 1 ^{-1} (D_3), \\
D_2=D(\pi+\theta-\alpha, \alpha, 2\alpha-\pi, \pi+\theta+\phi -2\alpha)=\gR 2 (D_3), \\
D_3=D(\alpha, \pi+\theta -\alpha,\theta,\phi). 
\end{cases}
\end{gather}

On the coordinates, the polyhedron $D_1$ is defined as 
\begin{equation*} 
D_1= \left\{ \textbf x=P(\textbf v) \colon
\begin{array}{l l}
\arg(x_ 1) \in (-\phi,0), & \arg(x_ 2) \in (0,\theta),\\
\arg(v_ 1) \in (0,\pi+\theta+\phi-2\alpha), & \arg(v_ 2) \in (0,2\alpha-\pi)
\end{array} \right\},
\end{equation*}
the polyhedron $D_2$ is
\begin{equation*} 
D_2= \left\{ \textbf y=P(\textbf w) \colon
\begin{array}{l l}
\arg(y_ 1) \in (-(\pi+\theta+\phi-2\alpha,0), & \arg(y_ 2) \in (0,2\alpha-\pi),\\
\arg(w_ 1) \in (0,\phi), & \arg(w_ 2) \in (0,\theta)
\end{array} \right\}
\end{equation*}
and the polyhedron $D_3$ is defined as 
\begin{equation*} 
D_3= \left\{ \textbf z=P(\textbf u) \colon
\begin{array}{l l}
\arg(z_ 1) \in (-\phi,0), & \arg(z_ 2) \in (0,\theta),\\
\arg(u_ 1) \in (0,\phi), & \arg(u_ 2) \in (0,\theta)
\end{array} \right\}.
\end{equation*}

Due to the fact that the matrix for $\gR 1$ is extremely simple, we will keep track only of three sets of coordinates, namely $\textbf{z}$-, $\textbf{w}$- and $\textbf{y}$-coordinates and use the relations in \eqref{eq:relationcoordxz}, \eqref{eq:relationcoorduw} and \eqref{eq:relationcoordyv} to give the other coordinates in term of these.

Then we can describe the polyhedron as follows.
\begin{equation*} 
D= \left\{ \textbf z=\gR 2(\textbf y)=\gR 2 P(\textbf w) \colon
\begin{array}{l l}
\arg(z_ 1) \in (-\phi,0), & \arg(z_ 2) \in (-\theta,\theta),\\
\arg(w_ 1) \in (0,\phi), & \arg(w_ 2) \in (-\theta,\theta), \\
\arg(y_ 1) \in (-\phi',\phi'), & \arg(y_ 2) \in (0,\theta')
\end{array} \right\},
\end{equation*}
with $\phi'=\pi+\theta+\phi-2\alpha$ and $\theta'=2\alpha-\pi$.

\begin{figure}
\centering
\includegraphics[width=1\textwidth]{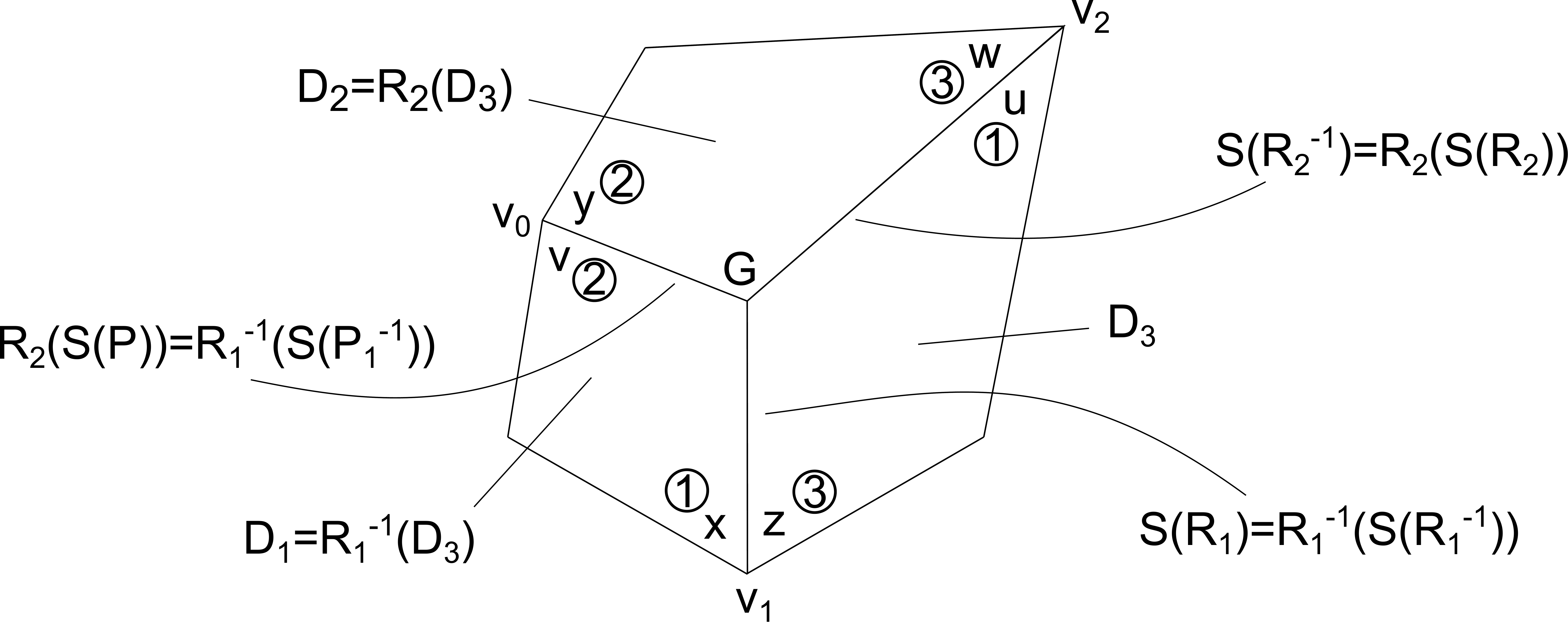}
\begin{quote}\caption{The interaction of the polyhedra and their coordinates. \label{fig:coord}} \end{quote}
\end{figure}

In Figure \ref{fig:coord} one can see how the polyhedra and the coordinates interact. 
The three polyhedra intersect pairwise in a side and all three have a common Giraud disc $G$. 
Passing from $\zzz$- to $\www$-coordinates changes the type of configuration from \conf j to \conf i within the same polyhedron $D_j$, where $i=j-2$, taken mod 3. 
The three special vertices $\textbf v_0$, $\textbf v_1$ and $\textbf v_2$ are the origin of one of the coordinates.

\subsubsection{Vertices of D}

The vertices of $D$ will be of three types. 
Some will come from $D_1$ and they will be called $\textbf x_i$, for $i=1, \dots, 14$, some will be the vertices of $D_2$ and we will denote them $\textbf y_i$, for $i=1, \dots, 14$ and finally there will be the vertices $\textbf z_i$'s for $i=1, \dots, 14$, coming from $D_3$.
Since the three polyhedra intersect there will be some vertices that are repeated. 
The following table describes all the vertices.
In the first column there will be the label we choose for the vertex, in the second, third and fourth column its name in $D_3$, $D_1$ and $D_2$ respectively (if there is one), and in the final columns we will record which coordinates have a "nice" form.

\begin{center}
\begin{longtable}{|c|c|c|c|c|c|c|c|c|c|}
\hline
$D$ & $D_3$ & $D_1$ & $D_2$ & $\arg z_1$ & $\arg z_2$ & $\arg w_1$ & $\arg w_2$ &$\arg y_1$ & $\arg y_2$ \\
\hline 
$\textbf v_0$ & & $\textbf x_2$ & $\textbf y_1$ 
& & & & & $y_1=0$& $y_2=0$ \\
$\textbf v_1$ & $\textbf z_1$ & $\textbf x_1$ &  
& $z_1=0$& $z_2=0$ & &&& \\
$\textbf v_2$ & $\textbf z_2$ &  & $\textbf y_2$ 
&& & $w_1=0$& $w_2=0$ && \\
$\textbf v_3$ & $\textbf z_3$ & $\textbf x_3$ & $\textbf y_5$
& 0& $z_2=0$ & 0& 0 & 0& $\theta'$ \\
$\textbf v_4$ & $\textbf z_4$ & $\textbf x_5$ & $\textbf y_4$
& 0& 0 & 0& $w_2=0$ & 0& 0 \\
$\textbf v_5$ & $\textbf z_5$ &  &  
& 0& $\theta$ & 0& $-\theta$ & &\\
$\textbf v_6$ & $\textbf z_6$ & $\textbf x_6$ & $\textbf y_{13}$ 
& $-\phi$& $z_2=0$ & 0& 0 & 0& $\theta'$ \\
$\textbf v_7$ & $\textbf z_7$ & $\textbf x_8$ & $\textbf y_{12}$ 
& $-\phi$&0 & $ \phi$& 0 & $y_1 =0$& $\theta'$ \\
$\textbf v_8$ & $\textbf z_8$ &  & $\textbf y_{14}$
& $-\phi$& $\theta$ & $w_1 =0$& 0 & $-\phi'$& $\theta'$ \\
$\textbf v_9$ & $\textbf z_9$ & $\textbf x_{12}$ &
& $z_1=0$& 0 & $\phi$& $-\theta$ & $\phi'$& 0\\
$\textbf v_{10}$ & $\textbf z_{10}$ & $\textbf x_{13}$ & $\textbf y_{10}$ 
&0& 0 & $\phi$& $w_2=0$ & 0& 0 \\
$\textbf v_{11}$ & $\textbf z_{11}$ & $\textbf x_{14}$ & $\textbf y_9$ 
& $-\phi$& 0 & $\phi$& 0& $y_1 =0$&0 \\
$\textbf v_{12}$ & $\textbf z_{12}$ &  & 
&$z_1=0$ & $\theta$ & $\phi$ & $-\theta$ && \\
$\textbf v_{13}$ & $\textbf z_{13}$ &  &  
& 0 & $\theta$ & 0 & $-\theta$ && \\
$\textbf v_{14}$ & $\textbf z_{14}$ &  & 
& $-\phi$ & $\theta$ & 0& $-\theta$ && \\
$\textbf v_{16}$ &  & $\textbf x_4$ & $\textbf y_3$ 
& 0 & $-\theta$ & 0 & $\theta$ &0 &$y_2=0$ \\
$\textbf v_{17}$ &  & $\textbf x_7$ &
& $-\phi$ & $-\theta$ && & $\phi'$ &$\theta'$ \\
$\textbf v_{18}$ &  & $\textbf x_9$ &  
& $z_1=0$ & $-\theta$ & & &$\phi'$ & 0 \\
$\textbf v_{19}$ &  & $\textbf x_{10}$ & 
& 0 & $-\theta$ & & &$\phi'$ &$y_2=0$ \\
$\textbf v_{20}$ &  & $\textbf x_{11}$ &
& $-\phi$& $-\theta$ & & &$\phi'$& $\theta'$ \\
$\textbf v_{21}$ &  &  & $\textbf y_6$
&& &0& $\theta$ & $-\phi'$ & $y_2=0$ \\
$\textbf v_{22}$ &  &  & $\textbf y_7$ 
&& &$\phi$ & $\theta$  & $-\phi'$ &0 \\
$\textbf v_{23}$ &  &  & $\textbf y_8$
&& &$w_1=0$ & $\theta$  & $-\phi'$ & $\theta'$ \\
$\textbf v_{24}$ &  &  & $\textbf y_{11}$ 
&& & $\phi$&$\theta$  & $-\phi'$ & 0 \\
\hline
\caption{The vertices of $D$}
\label{table:vertD}
\end{longtable}
\end{center}

This reflects how the $D_i$'s glue together. 
In particular, the polyhedra $D_1$ and $D_3$ glue along
\begin{equation}\label{eq:glueing13}
\{\im z_2=0 \}\cap D_3=\{\im e^{-i\theta} x_2=0 \} \cap D_1,
\end{equation}
while $D_2$ and $D_3$ are glued along 
\begin{equation}\label{eq:glueing23}
\{\im e^{-i\theta}u_2=0 \}\cap D_3=\{\im w_2=0 \} \cap D_2
\end{equation}
and $D_1$ and $D_2$ intersect along 
\begin{equation}\label{eq:glueing12}
\{\im v_1=0 \}\cap D_1=\{\im y_1=0 \} \cap D_2.
\end{equation}
Moreover, all three will intersect in the Giraud disc $G$ containing the ridge bounded by vertices $\textbf{v}_3, \textbf{v}_4, \textbf{v}_6, \textbf{v}_7, \textbf{v}_{10}$ and $\textbf{v}_{11}$ (see Figure \ref{fig:coord}). 

\begin{rk}\label{rk:samelines}
Using Table \ref{table:lineseq} one can obtain the equations of the complex lines for our three configurations and see that the following lines coincide:
\begin{enumerate}
\item \label{item:*0}$L_{*0}(\alpha, \pi+\theta-\alpha, \theta, \phi)=
L_{*0}(\pi+\theta-\alpha, \alpha, 2\alpha-\pi, \pi+\theta+\phi-2\alpha)=
L_{*0}(\alpha, \alpha, \theta, \phi)$,
\item \label{item:*3}$L_{*3}(\alpha, \pi+\theta-\alpha, \theta, \phi)=
L_{*3}(\pi+\theta-\alpha, \alpha, 2\alpha-\pi, \pi+\theta+\phi-2\alpha)=
L_{*2}(\alpha, \alpha, \theta, \phi)$,
\item \label{item:*1}$L_{*1}(\alpha, \pi+\theta-\alpha, \theta, \phi)=
L_{*2}(\pi+\theta-\alpha, \alpha, 2\alpha-\pi, \pi+\theta+\phi-2\alpha)=
L_{*1}(\alpha, \alpha, \theta, \phi)$.
\end{enumerate}
\end{rk}

\subsubsection{Sides and side pairing maps}

In view of applying Poincaré polyhedron theorem in Section \ref{sec:mainthm}, we need to analyse the sides of $D$ and explain how we have some maps pairing them. 

Clearly, the sides of $D$ will be the union of all sides in $D_i$, with $i=1,2,3$, except for the three sides along which two of the copies glue. 
Some of the sides combine to create a single larger side. 
Remembering \eqref{eq:Djohn}, the sides (illustrated in Figure \ref{fig:sides} with their side pairings) will be as follows.
\[
\begin{array}{llll}
S(J), & S(P), & S(R_1), & S(R_2), \\
S(J^{-1}), & S(P^{-1}), & S(R_1^{-1}), & S(R_2^{-1}), \\
R_1^{-1}S(J), & R_1^{-1}S(P), & R_1^{-1}S(R_1), & R_1^{-1}S(R_2), \\
R_1^{-1}S(J^{-1}), & R_1^{-1}S(P^{-1}), & R_1^{-1}S(R_1^{-1}), & R_1^{-1}S(R_2^{-1}), \\
R_2S(J), & R_2S(P), & R_2S(R_1), & R_2S(R_2), \\
R_2S(J^{-1}), & R_2S(P^{-1}), & R_2S(R_1^{-1}), & R_2S(R_2^{-1}).
\end{array}
\]

Now the glueing of the three polyhedra (see Equations \eqref{eq:glueing13}, \eqref{eq:glueing23} and \eqref{eq:glueing12}) tells us that 
\[
R_1^{-1}S(R_1^{-1})=S(R_1),\quad
R_2S(R_2)=S(R_2^{-1}), \quad
R_1^{-1}S(P^{-1})=R_2S(P),
\]
so these sides are now internal (see Figure \ref{fig:coord}).

\begin{figure}
\centering
\includegraphics[height=0.9\textheight]{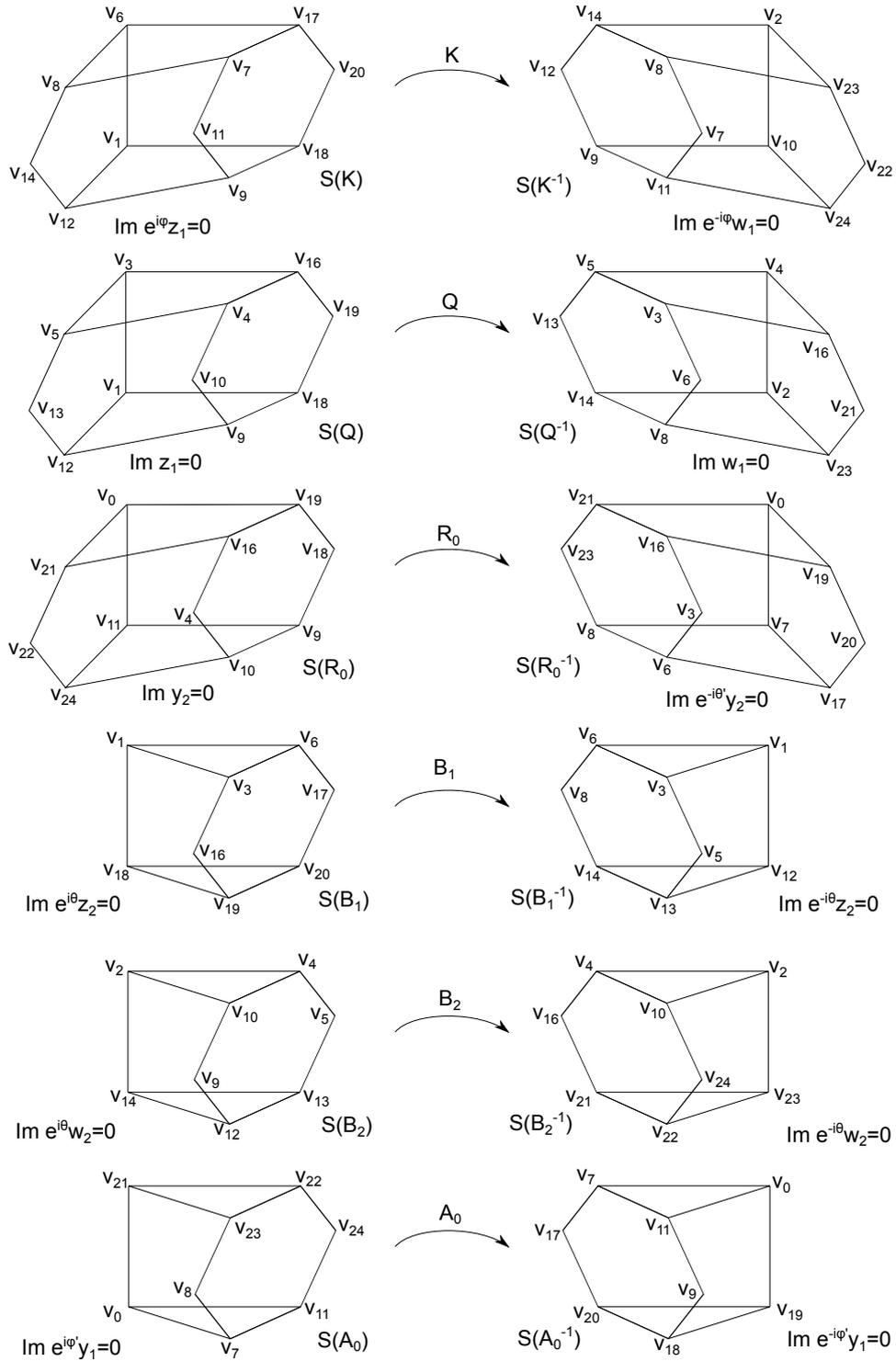}
\begin{quote}\caption{The sides of $D$. \label{fig:sides}} \end{quote}
\end{figure}

The side pairings will be obtained adapting to the union of the three polyhedra the equivalent on each $D_i$ of the side pairings in previous works (see Section 4.3 of \cite{livne}, Section 5.3 of \cite{boadiparker} and Section 8.3.1 of \cite{irene}). 
In other words, in each copy we need to consider $R_1$, $R_2$, $P$ and $J$ and adapt them to act on the sides of $D$.
We will describe all of them treating the $\textbf z$-coordinates as the main coordinates. 
In other words, we will give the matrix as applied to the $\textbf z$-coordinates of the point.

First consider $R_1$ and $R_2$. 
Since applying $R_2\cone$ to a point in its $\textbf x$-coordinates is equivalent to applying $R_1\ctwo$ to its $\textbf v=\textbf y$-coordinates, they combine to a single side pairing 
\[
\gR 1(\pi+\theta-\alpha, \alpha, 2\alpha-\pi, \pi+\theta+\phi-2\alpha)
=\begin{bmatrix}
1 & 0 & 0 \\
0 & -e^{2i\alpha} & 0 \\
0 & 0 & 1
\end{bmatrix}.
\]
This is the side pairing as applied on the $\textbf y$-coordinates. 
We will hence consider 
\begin{align*}
\gR 0&= \gR 2 \gR 1 \gR 2^{-1} =\gR 2(\pi+\theta-\alpha, \alpha, 2\alpha-\pi, \pi+\theta+\phi-2\alpha) \circ \\
&\circ \gR 1(\pi+\theta-\alpha, \alpha, 2\alpha-\pi, \pi+\theta+\phi-2\alpha) \circ \gR 2^{-1}(\alpha, \pi+\theta-\alpha, \theta,\phi),
\end{align*}
which includes the change of coordinates. 

Now consider $R_1 \cone$ and $R_1 \cthree$. 
The target side of the former coincides with the source side of the latter and is the (now internal) side $D_1 \cap D_3$.
This means that we can compose the two maps and have a new side pairing 
\[
\gRR 1=\gR 1(\alpha, \pi+\theta-\alpha, \theta,\phi) \gR 1(\alpha, \alpha, \theta, \phi)
=\begin{bmatrix}
1 & 0 & 0 \\
0 & e^{2i\theta} & 0 \\
0 & 0 & 1
\end{bmatrix}.
\]
Remark that even though it looks like this is the matrix we use when applying the map to a point in its $\textbf x$-coordinates, composing it with the change of coordinates from our coordinates (the $\textbf z$-coordinates) one gets that in terms of matrices
\[
\gRR 1 (\alpha, \pi+\theta-\alpha, \theta,\phi)=\gR 1^{-1}(\alpha, \alpha, \theta, \phi)\gRR 1(\alpha, \alpha, \theta, \phi)\gR 1(\alpha, \pi+\theta-\alpha, \theta,\phi)=\gRR 1 (\alpha, \alpha, \theta, \phi).
\] 

Similarly, the common side of $D_2$ and $D_3$ is the target side of $\gR 2 \cthree$ and the starting side of $\gR 2 \ctwo$. 
We can then define 
\[
\gRR 2=\gR 2(\pi+\theta-\alpha, \alpha, 2\alpha-\pi, \pi+\theta+\phi-2\alpha)\gR 2(\alpha, \pi+\theta-\alpha, \theta,\phi).
\]
The map $\gRR 2$ is already defined to act on the $\textbf z$-coordinates. 
As we said for $\gR 2$ and $\gR 1$, $\gRR 2$ acts as $\gRR 1$, but on the $\textbf u$-coordinates. 

The side pairings $P$ and $J$ related the $\zzz$- and $\www$-coordinates of the polyhedron, but the side pairing property relied on the fact that the source and target configurations were of the same type. 
Adapting this to our case means that we want to consider the maps relating $\textbf z$- and $\textbf w$-coordinates, $\textbf x$- and $\textbf u$-coordinates and $\textbf y$- and $\textbf v$-coordinates. 
The map relating $\textbf y$- and $\textbf v$-coordinates is the identity and it indeed maps the common side between $D_1$ and $D_2$ to itself.
Since this side is in the interior of $D$, we can ignore it. 
Composing the map obtained with $A_1 \ctwo$ to compute the equivalent of $J$ and applying the change of coordinates to our main coordinates, we obtain the side pairing 
\[
\gA 0= \gR 2\gA 1 \gR 2^{-1}.
\]

Now, we have
\[
\textbf w= P^{-1} \textbf y= P^{-1} \gR 2^{-1}\textbf z =Q^{-1} \textbf z
\]
and 
\[
\textbf u=\gR 2^{-1} \gR 1 ^{-1} \gR 1^{-1} x,
\]
which translates to the $\textbf z$-coordinates as $Q^{-1}$ again.
Then $Q=\gR 1 \gR 2 \gR 1$ will be our new side pairing. 
Moreover, we will consider $K=QA_1$.

Putting all this information together and remarking that $J^3=\id$, one gets that the side pairings are
\begin{eqnarray*}
K=JR_1=R_2J&:& R_1^{-1}S(J)\cup S(J)\longmapsto S(J^{-1})\cup R_2S(J^{-1}), \\
Q=P R_1=R_2 P&:& R_1^{-1}S(P)\cup S(P)
\longmapsto S(P^{-1})\cup R_2S(P), \\
\gR 0=R_1^{-1}R_2R_1=R_2R_1R_2^{-1}&:& R_1^{-1}S(R_2)\cup R_2S(R_1)
\longmapsto R_1^{-1}S(R_2^{-1})\cup R_2S(R_1^{-1}), \\
\gRR 1=\gR 1 \gR 1 &:& R_1^{-1}S(R_1) \longmapsto S(R_1^{-1}), \\
\gRR 2=\gR 2 \gR 2 &:& S(R_2) \longmapsto R_2S(R_2^{-1}), \\
\gA 0=R_1^{-1}J^{-1}J^{-1}R_2^{-1}&:& R_2S(J) \longmapsto R_1^{-1}S(J^{-1}).
\end{eqnarray*}

As mentioned for the general case, the sides are contained in bisectors. 
One can rewrite Lemma \ref{lemma:bisectors} for each copy and eliminate the inequalities related to the sides along which the polyhedra glue. 
Translating the inequalities on the right hand side into $\textbf{z}$-coordinates and giving all the $\textbf{n}_{*i}$ in terms of the configuration $(\alpha, \pi+\theta-\alpha, \theta, \phi)$ (using Remark \ref{rk:samelines}), we get the following lemma. 

\begin{lemma}\label{lemma:bisD}
We have
\begin{itemize}
\item $\im(z_1)\leq0$ if and only if \emph{
$\lvert \langle \textbf z, \textbf n_{*1} \rangle \rvert 
\leq \lvert \langle \textbf z, P^{-1}(\textbf n_{*3}) \rangle \rvert$},
\item $\im(e^{i \phi} z_ 1)\geq0$ if and only if \emph{
$\lvert \langle \textbf z, \textbf n_{*0} \rangle \rvert 
\leq \lvert \langle \textbf z, K^{-1}(\textbf n_{*0}) \rangle \rvert$},
\item $\im(e^{-i \theta} z_ 2)\leq0$ if and only if \emph{
$\lvert \langle \textbf z, \textbf n_{*3} \rangle \rvert 
\leq \lvert \langle \textbf z, \gRR 1(\textbf n_{*3}) \rangle \rvert$},
\item $\im(e^{i\theta}z_ 2)\geq0$ if and only if \emph{
$\lvert \langle \textbf z, \textbf n_{*3} \rangle \rvert 
\leq \lvert \langle \textbf z, \gRR 1^{-1}(\textbf n_{*3}) \rangle \rvert$},
\item $\im(e^{i \phi'} y_ 1)\geq0$ if and only if \emph{
$\lvert \langle \textbf z, \textbf n_{*0} \rangle \rvert 
\leq \lvert \langle \textbf z, K^{2}(\textbf n_{*0}) \rangle \rvert$},
\item $\im(y_ 2)\geq0$ if and only if \emph{
$\lvert \langle \textbf z, \textbf n_{*1} \rangle \rvert 
\leq \lvert \langle \textbf z, Q^{-1}\gRR 1(\textbf n_{*3}) \rangle \rvert$},
\item $\im(e^{-i \theta'} y_ 2)\leq0$ if and only if \emph{
$\lvert \langle \textbf z, \textbf n_{*3} \rangle \rvert 
\leq \lvert \langle \textbf z, \gRR 1^{-1}Q(\textbf n_{*1}) \rangle \rvert$},
\item $\im(e^{-i \phi'} y_ 1)\leq0$ if and only if \emph{
$\lvert \langle \textbf z, \textbf n_{*0} \rangle \rvert 
\leq \lvert \langle \textbf z, K^{-2}(\textbf n_{*0}) \rangle \rvert$},
\item $\im(w_ 1)\geq0$ if and only if \emph{
$\lvert \langle \textbf z, \textbf n_{*3} \rangle \rvert 
\leq \lvert \langle \textbf z, Q(\textbf n_{*1}) \rangle \rvert$},
\item $\im(e^{-i \phi} w_ 1)\leq0$ if and only if \emph{
$\lvert \langle \textbf z, \textbf n_{*0} \rangle \rvert 
\leq \lvert \langle \textbf z, K(\textbf n_{*0}) \rangle \rvert$},
\item $\im(e^{-i \theta} w_ 2)\leq0$ if and only if \emph{
$\lvert \langle \textbf z, \textbf n_{*1} \rangle \rvert 
\leq \lvert \langle \textbf z, \gRR 2(\textbf n_{*1}) \rangle \rvert$},
\item $\im(e^{i\theta}w_ 2)\geq0$ if and only if \emph{
$\lvert \langle \textbf z, \textbf n_{*1} \rangle \rvert 
\leq \lvert \langle \textbf z, \gRR 2^{-1}(\textbf n_{*1}) \rangle \rvert$}.
\end{itemize}
\end{lemma}

\section{Main theorem}\label{sec:mainthm}

In this section we will state and prove that $D$ (or a suitable modification of $D$) is a fundamental domain for Deligne-Mostow lattices with 2-fold symmetry as parametrised in Section \ref{sec:values}.
To do this we will use the Poincaré polyhedron theorem, in the form given in \cite{livne}, \cite{boadiparker} and \cite{irene}. 
It states that if one had a polyhedron $D$ and a set of side pairing maps $\{T_i\}$ satisfying certain conditions, then $D$ is a fundamental domain for the action of $\Gamma =\langle T_i \rangle$.
The main condition to check in this case is that suitable images of the polyhedron under the side pairing maps tessellate around the ridges. 
The theorem also provides a presentation for the group, with the side pairings as generators and relations coming from the tessellation conditions. 

\subsection{Main theorem}\label{sec:main}

We can now state that $D$ just defined or a suitable modification of it is a fundamental domain for the lattices we are considering.

\begin{theo} \label{thm:main}
Let $\Gamma$ be one of the subgroups of $PU(2,1)$ in Table \ref{table:values}.
Then a suitable modification of the polyhedron $D$ defined in Section \ref{sec:2fold} is a fundamental domain for $\Gamma$.
More precisely the fundamental domain is $D$ up to make some ridges collapse to a point when some parameters are degenerate (negative of infinite) according to the following table. 

\begin{center}
\begin{longtable}{|m{4cm}|m{2cm}|m{6cm}|}
\hline
Lattice & Deg. par. & Ridges collapsing \\
\hline
\emph{(4,4,6), (4,4,5)} & & \\
\hline
\emph{(3,4,4), (2,4,3), (3,3,4)} & $l'$, $d$ & $F(\gA 0,\gRR 2^{-1}), F(\gRR 2,\gRR 1^{-1})$, $F(\gA 0^{-1},\gRR 1)$, $F(Q,Q^{-1})$ \\
\hline
\emph{(2,6,6)} & $l$, $d$ & $F(Q,Q^{-1}), F(K,\gR 0^{-1})$, \ $F(K^{-1},\gR 0)$ \\
\hline
\emph{(2,3,3)} & $l$, $l'$, $d$ & $F(Q,Q^{-1}), F(K,\gR 0^{-1})$, $F(K^{-1},\gR 0), F(\gA 0,\gRR 2^{-1})$, $F(\gRR 2,\gRR 1^{-1})$, $F(\gA 0^{-1},\gRR 1)$ \\
\hline
\emph{(3,3,3), (4,4,3), (6,6,3)} & $k'$, $l'$, $d$ & $F(\gA 0,\gA 0^{-1}), F(Q,Q^{-1})$, $F(\gRR 2,\gRR 1^{-1})$, $F(\gA 0^{-1},\gRR 1)$ \\
\hline
\emph{(2,3,3)} & $k'$ & $F(\gA 0,\gA 0^{-1})$ \\
\hline
\end{longtable}
\end{center}

Moreover, a presentation for $\Gamma$ is given by 
\[
\Gamma=\left\langle 
\begin{array}{l l} 
K,Q,\gRR 1, \\
\gRR 2,\gR 0,\gA 0
\end{array}
\colon
\begin{array}{l l l l}
\gRR 1^p=\gRR 2^p=\gR 0^{p'}=\gA 0^{k'}=(Q^{-1}K)^k=
(\gR 0K)^l=I, \\
(\gA 0\gRR 2\gRR 1)^{l'}=Q^{2d}=I, \quad Q=\gRR 1\gR 0=\gR 0\gRR 2=\gRR 2^{-1}Q\gRR 1, \\ 
\gR 0^{-1}\gA 0\gR 0=\gA 0=K^{-2}, \quad \gRR 2 K=K \gRR 1
\end{array} \right\rangle,
\]
where the first two relations hold as long as the order is finite and positive.
\end{theo}

The reason for the ridges to collapse to a point (except for $k'$, which is treated in Section \ref{sec:kneg}) relies on the combinatorial structure of the polyhedron as explained in Section \ref{sec:degen}.
More precisely,
\begin {itemize}
\item First consider the case when $d <0$ or $d=\infty$.
By definition (see \eqref{eq:p,p',k,k',l,l',d}), this is equivalent to say that $\pi-\alpha-\theta \leq 0$.
Remembering Proposition \ref{prop:gencollapse} and using the notation of Remark \ref{rk:samelines}, one can see that the vertices on $L_{*0}$ collapse when $\pi-\alpha-\theta \leq 0$. 
Since these three vertices form the ridge $F(Q,Q^{-1})$, this ridge collapses when $d <0$ or $d=\infty$. 
\item Similarly, when $l<0$ or $l=\infty$, by definition, we have $\alpha-\theta-\phi \leq 0$.
Now the vertices on $L_{*3}$ collapse when  $\alpha-\theta-\phi \leq 0$ and so do the ones on $L_{*1}$.
Since $F(K^{-1},\gR 0)$ is formed of the vertices contained in $L_{*3}$ and $F(K,\gR 0^{-1})$ of the ones contained in $L_{*1}$, they degenerate when $l<0$ or $l=\infty$.
\item Now assume $l'<0$ or $l'=\infty$, i.e. $\pi-\alpha-\phi \leq 0$.
By Proposition \ref{prop:gencollapse}, the vertices on $L_{*2}(\alpha,\pi+\theta-\alpha,\theta,\phi)$, $L_{*3}(\alpha, \alpha,\theta,\phi)$ and $L_{*1}(\pi+\theta-\alpha, 2\alpha-\pi,\pi+\theta+\phi-2\alpha)$ all degenerate when $\pi-\alpha-\phi \leq 0$. 
Then the claim of the theorem follows from the fact that $F(\gRR 1,\gA 0^{-1})$, $F(\gRR 1^{-1},\gRR 2)$ and $F(\gRR 2^{-1},\gA 0)$ are bounded by the vertices contained in $L_{*3}(\alpha, \alpha,\theta,\phi)$, $L_{*2}(\alpha,\pi+\theta-\alpha,\theta,\phi)$ and $L_{*1}(\pi+\theta-\alpha, 2\alpha-\pi,\pi+\theta+\phi-2\alpha)$ respectively.
\item Finally, the case of $k'$ negative is treated in Section \ref{sec:kneg}.
\end{itemize}

An alternative presentation for the lattices can be obtained by remembering that $K=Q \gA 1$ and substituting $Q=\gRR 1 \gR 0$, $K=\gRR 1 \gR 0 \gA 1$, $\gRR 2= \gR 0^{-1} \gRR 1 \gR 0$ and $\gA 0=(\gRR 1 \gR 0 \gA 1)^{-2}$. 
Then
\[
\Gamma=\left\langle 
\begin{array}{l} 
\gRR 1, \gR 0,\gA 1
\end{array}
\colon
\begin{array}{l l l l}
\gRR 1^p=\gR 0^{p'}=(\gRR 1 \gR 0 \gA 1)^{2k'}=\gA 1^k=
(\gR 0 \gRR 1 \gR 0 \gA 1)^l=I, \\
(\gA 1\gR 0)^{2l'}=(\gRR 1 \gR 0)^{2d}=I, \quad \br 4 {\gRR 1}{\gR 0}, \\ 
\br 2 {(\gRR 1 \gR 0 \gA 1)^{-2}}{\gR 0}, \quad \br 2 {\gA 1}{\gRR 1}
\end{array} \right\rangle,
\]
where, following \cite{dppcommens}, $\br i T S$ is the braid relation of length $i$ on $T$ and $S$. 

\subsection{Volume}

The volume of the quotient is a multiple of the orbifold Euler characteristic $\chi( \hc / \Gamma)$.
This multiple is $\frac{8\pi^2}{3}$ when the holomorphic sectional curvature is normalised to -1.
The orbifold Euler characteristic is calculated by taking the alternating sum of the reciprocal of the order or the stabilisers of each orbit of cell.

In the following table we list the orbits of facets by dimension, calculate the stabiliser of the first element in the orbit and give its order.
Later, we will explain how the table changes when considering the degenerations of $D$.

\begin{longtable}{|c|c|c|}
\hline
Orbit of the facet & Stabiliser & Order \\
\hline 
$\textbf{v}_1, \textbf v_2$ & $\langle \gA 1,\gRR 1 \rangle$ & $kp$ \\
$\textbf{v}_3, \textbf v_4$ & $\langle Q^2,\gRR 1 \rangle$ & $pd$ \\
$\textbf{v}_{16}, \textbf v_5$ & $\langle Q^2,\gR 0 \rangle$ & $p'd$ \\
$\textbf{v}_6, \textbf v_{10}$ & $\langle \gR 0K,\gRR 1 \rangle$ & $pl$ \\
$\textbf{v}_7, \textbf v_{11}$ & $\langle \gR 0K,\gA 0 \rangle$ & $k'l$ \\
$\textbf{v}_{8}, \textbf v_{9}, \textbf{v}_{17}, \textbf v_{24}$ & $\langle QK^{-1}, \gR 0K \rangle$ & $kl$ \\
$\textbf{v}_{18}, \textbf v_{14}, \textbf v_{20}, \textbf{v}_{22}, \textbf v_{23}, \textbf v_{12}$ & $\langle \gA 0\gRR 2\gRR 1, \gA 1 \rangle$ & $l'k$ \\
$\textbf{v}_{19}, \textbf v_{13}, \textbf z_{21}$ & $\langle \gA 0\gRR 2\gRR 1,\gR 0 \rangle$ & $p'l'$ \\
$\textbf v_0$ & $\langle \gR 0,\gA 0 \rangle$ & $k'p'$ \\
\hline
$\gamma_{1,3}, \gamma_{2,4}$ & $\langle \gRR 1 \rangle$ & $p$\\
$\gamma_{1,6}, \gamma_{2,10}$ & $\langle \gRR 1 \rangle$ & $p$\\
$\gamma_{1,12}, \gamma_{2,23}, \gamma_{2,14}, \gamma_{1,18}$ & $\langle \gA 1 \rangle$ & $k$ \\
$\gamma_{3,5}, \gamma_{4,16}, \gamma_{4,5}, \gamma_{3,16}$ & $\langle Q^2 \rangle$ & $d$ \\
$\gamma_{3,6}, \gamma_{4,10}$ & $\langle \gRR 1 \rangle$ & $p$\\
$\gamma_{5,13}, \gamma_{16,19}, \gamma_{16,21}$ & $\langle \gR 0 \rangle$ & $p'$\\
$\gamma_{6,8}, \gamma_{10,24}, \gamma_{9,10}, \gamma_{6,17}$ & $\langle \gR 0K \rangle$ & $l$ \\
$\gamma_{7,8}, \gamma_{11,24}, \gamma_{9,11}, \gamma_{7,17}$ & $\langle \gR 0K \rangle$ & $l$ \\
$\gamma_{7,11}$ & $\langle K \rangle$ & $2k'$ \\
$\gamma_{7,15}, \gamma_{11,15}$ & $\langle \gA 0 \rangle$ & $k'$ \\
$\gamma_{8,14}, \gamma_{22,24}, \gamma_{17,20}, \gamma_{9,18}, \gamma_{23,8}, \gamma_{9,12}$ & $\langle \gA 1 \rangle$ & $k$\\
$\gamma_{12,13}, \gamma_{21,22}, \gamma_{18,19}, \gamma_{21,23}, \gamma_{19,20}, \gamma_{13,14}$ & $\langle \gRR 1 \gA 0 \gRR 2 \rangle$ & $l'$\\
$\gamma_{12,14}, \gamma_{22,23}, \gamma_{18,20}$ & $\langle \gRR 2^{-1} K \rangle$ & $2l'$\\
$\gamma_{15,19}, \gamma_{15,21}$ & $\langle \gR 0 \rangle$ & $p'$\\
\hline
$\begin{array}{l l}
F(K,Q) , &  F(K^{-1}, Q^{-1})
\end{array}$ 
& $\langle \gA 1\rangle$ & $k$ \\
$\begin{array}{l l}
F(K^{-1},\gR 0), & F(K,\gR 0^{-1})
\end{array}$ 
& $\langle K\gR 0 \rangle$ & $l$ \\
$F(\gR 0, \gR 0^{-1}) $
& $\langle \gR 0 \rangle$ & $p'$ \\
$F(Q,Q^{-1}) $
& $\langle Q \rangle$ &  $2d$ \\
$\begin{array}{l l l}
F(\gRR 1,\gA 0^{-1}) , & F(\gRR 1^{-1},\gRR 2), & F(\gRR 2^{-1},\gA 0)
\end{array}$
& $\langle \gRR 1 \gA 0 \gRR 2 \rangle$ & $l'$ \\
$F(\gRR 1, \gRR 1^{-1}) $
& $\langle \gRR 1 \rangle$ & $p$ \\
$F(\gRR 2, \gRR 2^{-1}) $
& $\langle \gRR 2 \rangle$ &  $p$ \\
$F(\gA 0, \gA 0^{-1}) $
& $\langle A'_1\rangle$ &  $k'$ \\
$\begin{array}{l l l l}
F(K,\gRR 1), & F(K,\gRR 1^{-1}), & F(K^{-1},\gRR 2^{-1}), & F(K^{-1},\gRR 2) 
\end{array}$
& 1 & 1 \\
$\begin{array}{l l l l}
F(\gRR 1,Q) , & F(\gRR 2,Q^{-1}), & F(\gRR 2^{-1},Q^{-1}) , & F(\gRR 1^{-1},Q)
\end{array}$
& 1 & 1 \\
$\begin{array}{l l l l}
F(\gA 0,\gR 0) , & F(\gA 0^{-1},\gR 0) , & F(\gA 0^{-1},\gR 0^{-1}) , & F(\gA 0,\gR 0^{-1})
\end{array}$
& 1 & 1 \\
$\begin{array}{l l l}
F(K,K^{-1}), & F(K^{-1},\gA 0), & F(K,\gA 0^{-1}) 
\end{array}$
& 1 & 1 \\
$\begin{array}{l l l}
F(\gRR 1,\gR 0^{-1}) , & F(\gRR 1^{-1},Q^{-1}), & F(Q,\gR 0) 
\end{array}$
& 1 & 1 \\
$\begin{array}{l l l}
F(\gR 0^{-1},Q^{-1}), & F(Q,\gRR 2), & F(\gRR 2^{-1},\gR 0)
\end{array}$
& 1 & 1 \\
\hline 
$S(K),S(K^{-1})$ & 1 & 1 \\
$S(Q),S(Q^{-1})$ & 1 & 1 \\
$S(\gRR 2),S(\gRR 2^{-1})$ & 1 & 1 \\
$S(\gRR 1),S(\gRR 1^{-1})$ & 1 & 1 \\
$S(\gR 0),S(\gR 0^{-1})$ & 1 & 1 \\
$S(\gA 0),S(\gA 0^{-1})$ & 1 & 1 \\
\hline 
$D$ & 1 & 1 \\
\hline
\caption{The stabilisers when all values are positive and finite.}
\label{table:nothing}
\end{longtable}

Then the orbifold Euler characteristic of $D$ is given by 
\begin{align}
\chi (\hc / \Gamma)&=
\frac{1}{kp}+\frac{1}{pd}+\frac{1}{dp'}
+\frac{1}{pl}+\frac{1}{k'l}+\frac{1}{kl}
+\frac{1}{l'k}+\frac{1}{p'l'}+
\frac{1}{k'p'} \nonumber \\
&-\frac{1}{p}-\frac{1}{p}-\frac{1}{k}-\frac{1}{d}-\frac{1}{p}
-\frac{1}{p'}-\frac{1}{l}-\frac{1}{l}
-\frac{1}{2k'}-\frac{1}{k'}-\frac{1}{k}-\frac{1}{l'}
-\frac{1}{2l'}-\frac{1}{p'} \nonumber \\
&+\frac{1}{k}+\frac{1}{l}+\frac{1}{p'}+\frac{1}{2d}+\frac{1}{l'}
+\frac{1}{p}+\frac{1}{p}+\frac{1}{k'}+1+1+1+1+1+1 \nonumber \\
&-1-1-1-1-1-1+1
\end{align}
and the volume is $\frac{8 \pi^2}{3} \chi(\hc / \Gamma)$.

Now we need to explain how to modify the table when calculating the orbifold Euler characteristic for one of the degenerations of $D$. 

\begin{itemize}
\item First consider the case when $d$ is  negative or infinite. 
Then the vertices $\textbf v_3$, $\textbf v_4$, $\textbf v_5$ and $\textbf v_{16}$ collapse to a single point. 
This means that the two orbits containing  them will collapse to only one orbit. 
The new vertex is stabilised by $\langle \gRR 1, \gR 0 \rangle$ and we need to calculate its order. 
This is similar to the proof of Proposition 2.3 of \cite{dppcommens} (adapting the argument to complex reflections with different orders) and to proof of 4.4, 4.5 and 4.6 in \cite{survey}. 
Now, $\gR 0$ has eigenvalues $e^{i \theta'}, 1, 1$, while $\gRR 1$ has eigenvalues $e^{2 i \theta},1,1$. 
In other words, remembering $\theta'= \frac{2\pi}{p'}$ and $\theta=\frac{\pi}{p}$, $\gR 0$ and $\gRR 1$ have eigenvalues $e^{2 i \pi /p'},1,1$ and $e^{2 i \pi/p},1,1$ respectively. 
Now consider $\gRR 1\gR 0$. It has eigenvalues $1, e^{i(\alpha+\theta)}, -e^{i(\alpha+\theta)}$, which we can write as $1, e^{i(\frac{\pi}{p'}+\frac{\pi}{p} +\frac{\pi}{2})}, e^{i(\frac{\pi}{p'}+\frac{\pi}{p} -\frac{\pi}{2})}$ because $\theta'=2\alpha-\pi$. 
In this way the part acting on the sphere orthogonal to the common eigenspace is in $SU(2)$. 
This means that $\langle \gR 0, \gRR 1 \rangle$ is a central extension of a $(2,p,p')$-triangle group.
Remembering that a $(2,a,b)$-triangle group has order $\frac{4ab}{2a+2b-ab}$ and the definition of the parameters \eqref{eq:p,p',k,k',l,l',d}, the order of the triangle group is $-2d$.
Since $\pi-\alpha-\theta=\frac{\pi}{d}$, the eigenvalues of $(\gR 0\gA 1)^2$ are $e^{\frac{2\pi}{d}}, e^{\frac{2\pi}{d}}, 1$ and hence the order of the centre is $-d$. 
So the order of the stabiliser is $2d^2$. 
Moreover, the line of the table corresponding to the edges between these three points (so the line of the orbit of $\gamma_{3,5}$) needs to be eliminated and so does the line corresponding to the ridge $F(Q,Q^{-1})$.
\item Now consider the case of $l'$ negative or infinite. 
We have three triples of vertices collapsing to the three vertices $\vertr {12}{13}{14}$, $\vertr {18}{19}{20}$ and $\vertr {21}{22}{23}$, where $\vertr i j k$ is obtained collapsing vertices $\textbf v_i$, $\textbf v_j$ and $\textbf v_k$. 
They are in a unique orbit and $\vertr {18}{19}{20}$ is stabilised by $\langle \gR 0, Q^{-1}K \rangle=\langle \gR 0, \gA 1 \rangle$. 
We need to calculate its order.
Now, $\gR 0$ has eigenvalues $e^{i \theta'}, 1, 1$, while $\gA 1$ has eigenvalues $e^{2 i \phi},1,1$. 
In other words, remembering $\theta'= \frac{2\pi}{p'}$ and $\phi=\frac{\pi}{k}$, $\gR 0$ and $\gA 1$ have eigenvalues $e^{2 i \pi /p'},1,1$ and $e^{2 i \pi/k},1,1$ respectively. 
Now consider $\gR 0\gA 1$. It has eigenvalues $1, e^{i(\alpha+\phi)}, -e^{i(\alpha+\phi)}$, which we can write as $1, e^{i(\frac{\pi}{p'}+\frac{\pi}{k} +\frac{\pi}{2})}, e^{i(\frac{\pi}{p'}+\frac{\pi}{k} -\frac{\pi}{2})}$. 
This means that $\langle \gR 0, \gA 1 \rangle$ is a central extension of a $(2,p',k)$-triangle group, which has order $\frac{4p'k}{2p'+2k-p'k}=-2l'$.
Since $\alpha+\phi-\pi=\frac{\pi}{l'}$, the eigenvalues of $(\gR 0\gA 1)^2$ are $e^{\frac{2\pi}{l'}}, e^{\frac{2\pi}{l'}}, 1$ and hence the order of the centre is $-l'$. 
This means that the order of $\langle \gR 0, \gA 1 \rangle$ is $2l'^2$.
Moreover, the two lines of the table corresponding to edges between the three collapsing points need to be eliminated. 
In other words, the lines of the orbits of $\gamma_{12,13}$ and $\gamma_{12,14}$ disappear from the table, together with the orbit of the three ridges that collapse. 
\item Now let us consider the parameter $l$. 
From Table \ref{table:valuespar} one can see that it is never negative. 
The only degeneration hence comes when it is infinite. 
This means that the two vertices obtained by triples collapsing are on the boundary and hence their stabiliser will have infinite order.
So the orbit of these two vertices disappears in the calculation of the orbifold Euler characteristic. 
Similarly, the two orbits of edges between collapsing vertices disappear from the calculation (the orbits of $\gamma_{6,8}$ and $\gamma_{7,8}$) and so does the orbit containing the two ridges that collapse to the two new points on the boundary. 
\item When $k'$ is negative, the vertices $\textbf v_0$, $\textbf v_7$ and $\textbf v_{11}$ collapse to a point (see Section \ref{sec:kneg}). 
This means that the two orbits of these three points collapse to a single one. 
It is easy to see that the new point is stabilised by $K$, $\gA 0$ and $\gR 0$, so the stabiliser is $\langle \gR 0, K \rangle$. 
We now need to calculate the order of this group.
Since $K^2=\gA 0^{-1}$ and $\gA 0$ commutes with $\gR 0$, the centre is generated by $K^2$, which has order $-k'$. 
Now, we know that $\gR 0 K$ has order $l$, so $\langle \gR 0, K \rangle$ modulo the centre would is a $(2,p',l)$-triangle group, which has order $-2k'$. 
So the order of $\langle \gR 0, K \rangle$ is $2k'^2$.
Moreover, the lines of the table corresponding to the two orbits of edges between these three points (i.e. the orbit of $\gamma_{7,11}$ and $\gamma_{7,0}$) disappear in the calculation and so does the line relative to $F(\gA 0,\gA 0^{-1})$. 
\end{itemize}

The orbifold Euler characteristic calculated with the modification of Table \ref{table:nothing} is coherent with the commensurability theorems we know between Deligne-Mostow lattices in $PU(2,1)$. 
The Table below summarises the values found in relation with the ones previously known.
The first part contains lattices that are commensurable according to Corollary 3.9 in \cite{survey}, which is Corollary 10.18 of \cite{dmbook} and have index 6. 
The second block correspond to commensurability stated in Theorem 3.10 in \cite{survey} and they also have index 6.
The third part of the table contains lattices which are commensurable according to Theorem 3.8 in \cite{survey}, which is Theorem 10.6 in \cite{dmbook}. 
They have index 2, except for the first one in the list, where the extra term in the index is given by the fact that the theorem doesn't consider the 3-fold symmetry of the lattice. 
Finally, Proposition 7.10 of \cite{dppcommens} shows that the Thompson group $E_2$ when $p=4$ is (conjuguate to) a subgroup of index 3 in the Deligne-Mostow group where $\mu=(3,3,5,6,7)/12$, which is exactly our $(3,4,4)$.

\begin{center}
\begin{tabular}{r r| l l || r r | l l }
Lattice & Volume & Volume & Lattice & Lattice & Volume & Volume & Lattice\\[7pt]
\toprule
(6,6,3) & $\frac{1}{3\cdot 2^2}$ & $\frac{1}{3^2 \cdot 2^3}$ & (6,2) &
(10,10,5) & $\frac{3}{2^2 \cdot 5}$ & $\frac{1}{2^3 \cdot 5}$ & (10,2) \\[7pt]
(12,12,6) & $\frac{7}{2^4 \cdot 3}$ & $\frac{7}{2^5 \cdot 3^2}$ & (12,2) &
(18,18,9) & $\frac{13}{2^2 \cdot 3^3}$ & $\frac{13}{2^3 \cdot 3^4}$ & (18,2) \\[3pt]
\midrule
(4,4,3) & $\frac{1}{3\cdot 2^2}$ & $\frac{1}{3^2 \cdot 2^3}$ & (4,3) &
(4,4,5) & $\frac{11 \cdot 3^3}{2^4 \cdot 5^2}$ & $\frac{11 \cdot 3}{2^5 \cdot 5^2}$ & (4,5) \\[7pt]
(4,4,6) & $\frac{13}{2^4 \cdot 3}$ & $\frac{13}{2^5 \cdot 3^2}$  & (4,6) &&&&\\[3pt]
\midrule
(2,6,6) & $\frac{1}{2^3}$ & $\frac{1}{3 \cdot 2^2}$ & (6,6) &
(2,3,3) & $\frac{1}{3 \cdot 2^3}$ & $\frac{1}{3 \cdot 2^2}$  & (3,3,3)\\[7pt]
(3,3,4) & $\frac{7}{2^4 \cdot 3}$ & $\frac{7}{2^5 \cdot 3}$  & (2,4,3) &&&&\\[3pt]
\midrule
(3,4,4) & $\frac{17}{3\cdot 2^5}$ & $\frac{17}{2^5}$  & $\mathcal T (4, E_2)$ &&&&\\[3pt]
\bottomrule
\end{tabular}
\end{center}

\subsection{Cycles}\label{sec:cycles}

The cycles given by Poincaré polyhedron theorem are the following. 
\begin{longtable}{c}
$F(K,Q) \xrightarrow{K} F(K^{-1}, Q^{-1}) \xrightarrow{Q^{-1}}
F(K,Q)$, 
\\
$F(K^{-1},\gR 0) \xrightarrow{\gR 0} F(K,\gR 0^{-1}) \xrightarrow{K} F(K^{-1},\gR 0)$, 
\\
$F(\gRR 1,\gA 0^{-1}) \xrightarrow{\gRR 1} F(\gRR 1^{-1},\gRR 2) \xrightarrow{\gRR 2}
F(\gRR 2^{-1},\gA 0) \xrightarrow{\gA 0} F(\gRR 1,\gA 0^{-1})$, 
\\
$F(\gR 0,\gR 0^{-1}) \xrightarrow{\gR 0} F(\gR 0,\gR 0^{-1})$,
\\
$F(Q,Q^{-1}) \xrightarrow{Q} F(Q,Q^{-1})$, 
\\
$F(\gRR 1,\gRR 1^{-1}) \xrightarrow{\gRR 1} F(\gRR 1,\gRR 1^{-1})$,
\\
$F(\gRR 2,\gRR 2^{-1}) \xrightarrow{\gRR 2} F(\gRR 2,\gRR 2^{-1})$,
\\
$F(\gA 0,\gA 0^{-1}) \xrightarrow{\gA 0} F(\gA 0,\gA 0^{-1})$,
\\
$F(K,\gRR 1) \xrightarrow{\gRR 1}
F(K,\gRR 1^{-1})\xrightarrow{K} F(K^{-1},\gRR 2^{-1}) \xrightarrow{\gRR 2^{-1}}
F(K^{-1},\gRR 2) \xrightarrow{K^{-1}} F(K,\gRR 1)$,
\\
$F(\gRR 1,Q) \xrightarrow{Q} F(\gRR 2,Q^{-1}) \xrightarrow{\gRR 2}
F(\gRR 2^{-1},Q^{-1}) \xrightarrow{Q^{-1}} F(\gRR 1^{-1},Q) \xrightarrow{\gRR 1^{-1}}
F(\gRR 1,Q)$,
\\
$F(\gA 0,\gR 0) \xrightarrow{\gA 0} F(\gA 0^{-1},\gR 0) \xrightarrow{\gR 0} F(\gA 0^{-1},\gR 0^{-1}) \xrightarrow{\gA 0^{-1}}
F(\gA 0,\gR 0^{-1}) \xrightarrow{\gR 0^{-1}} 
F(\gA 0,\gR 0)$,
\\
$F(K,K^{-1}) \xrightarrow{K} F(K^{-1},\gA 0) \xrightarrow{\gA 0}
F(K,\gA 0^{-1}) \xrightarrow{K} F(K,K^{-1})$,
\\
$F(\gRR 1,\gR 0^{-1}) \xrightarrow{\gRR 1} F(\gRR 1^{-1},Q^{-1}) \xrightarrow{Q^{-1}}
F(Q,\gR 0) \xrightarrow{\gR 0} F(\gRR 1, \gR 0^{-1})$,
\\
$F(\gR 0^{-1},Q^{-1}) \xrightarrow{Q^{-1}}
F(Q,\gRR 2) \xrightarrow{\gRR 2} F(\gRR 2^{-1},\gR 0) \xrightarrow{\gR 0} F(\gR 0^{-1},Q^{-1})$.
\end{longtable}

The cycles give the following transformations, where $\ell$ determines the power of $T$ which fixes the ridge pointwise and $\ell m$ is the order of $T$.
Note that for all of the 2-fold symmetry values that we are considering, $k, k', p, p', l, l'$ and $d$ are all integers (positive or negative). 

\begin{longtable}{|c|c|c||c|c|c|}
\hline
Cycle transformation $T$
& $\ell$ & $m$
& Cycle transformation $T$
& $\ell$ & $m$ \\
\hline
$Q^{-1}K$ 
& 1 & $k$ 
& $\gA 0$ 
& 1 & $k'$ \\

$\gR 0$
& 1 & $p'$ 
& $\gRR 1\gA 0\gRR 2=(\gRR 2^{-1}K)^2$
& 1 & $l'$ \\

$\gRR 2$ 
& 1 & $p$ 
& $\gRR 1$
& 1 & $p$ \\

$Q$ 
& 2 & $d$ 
& $\gR 0K$
& 1 & $l$ \\

$\gR 0Q^{-1}\gRR 1=\id$ 
& 1 & 1 
& $\gRR 2 Q^{-1}\gR 0=\id$
& 1 & 1 \\

$\gRR 1K^{-1}\gRR 2^{-1}K=\id$
& 1 & 1 
& $\gRR 1^{-1}Q^{-1}\gRR 2Q=\id$
& 1 & 1 \\

$\gA 0\gR 0^{-1}\gA 0^{-1}\gR 0=\id$
& 1 & 1 
& $K\gA 0K=\id$
& 1 & 1 \\
\hline
\caption{The cycle transformations and their orders.}
\label{table:cycletransf}
\end{longtable}

When the order of a cycle transformation is negative, we know from Section \ref{sec:main} that the corresponding ridge collapses to a point and so the transformation is a complex reflection to a point. 
When the order is $\infty$, then the cycle transformation is parabolic.

\subsection{Tessellation}

The proof of Theorem \ref {thm:main} consists in proving that $D$ and our side pairings satisfy the hypothesis of Poincaré polyhedron theorem.  
This is done in the same way as in \cite {livne}, \cite {boadiparker} and \cite {irene}. 
We will include some proofs of the tessellation condition, since it is the hardest to prove. 
We will divide the ridges in three groups. 
Looking at the structure of sides in Figure \ref{fig:sidecombinatorics}, one can see that the ridges are contained in either a Giraud disc, a Lagrangian plane or a complex line. 
We will include the proof of the tessellation condition for one ridge from each type.

\paragraph{Ridges contained in a Giraud disc.}

The ridges contained in a Giraud disc are $F(K,K^{-1})$, $F(K,\gA 0^{-1})$, $F(\gA 0,K^{-1})$, $F(\gRR 1,\gR 0^{-1})$, $F(\gRR 1^{-1},Q^{-1})$, $F(Q,\gR 0)$, $F(\gR 0^{-1},Q^{-1})$, $F(Q,\gRR 2)$ and $F(\gRR 2^{-1},\gR 0)$.
To prove the tessellation condition for them, we will use Lemma \ref{lemma:bisD}. 
The proof follows proofs of Propositions 4.5 and 4.7 of \cite{livne}, Proposition 5.3 (first part of the proof) in \cite{boadiparker} and Proposition 8.7 of \cite{irene}.

\begin{prop}
The polyhedra $D$, $K(D)$ and $K\gA 0(D)=K^{-1}(D)$ tessellate around the ridge $F(K,K^{-1})$.
\end{prop}

\begin{proof}
Take $\textbf z \in D$. 
By the second point of Lemma \ref{lemma:bisD}, $\textbf z$ is closer to $\textbf n_{*0}$ than to $K^{-1}(\textbf n_{*0})$. 
By the tenth point of the lemma, it is closer to $\textbf n_{*0}$ than to $K(\textbf n_{*0})$. 
Similarly, take a point $\textbf z \in K(D)$. 
This means that $K^{-1}(\textbf z) \in D$.
By the second point of the lemma applied to $K^{-1}(\textbf z) $, $\textbf z$ is closer to $K(\textbf n_{*0})$ than to $\textbf n_{*0}$. 
By the eighth point of the lemma, it is closer to $K(\textbf n_{*0})$ than to $K^{-1}(\textbf n_{*0})$. 
Finally, take a point $\textbf z \in K^{-1}(D)$. 
This means that $K(\textbf z) \in D$.
By the fifth point of the lemma applied to $K(\textbf z)$, $\textbf z$ is closer to $K^{-1}(\textbf n_{*0})$ than to $K(\textbf n_{*0})$. 
By the tenth point of the lemma, it is closer to $K^{-1}(\textbf n_{*0})$ than to $\textbf n_{*0}$. 

This clearly implies that the three images are disjoint and since $F(K,K^{-1})$ is defined by $\im (e^{i \phi} z_1)=0$ and $\im(e^{-i \phi}w_1)=0$, a small enough neighbourhood of the ridge is covered by the three images.
\end{proof}

\paragraph{Ridges contained in a Lagrangian plane.}

The ones contained in a Lagrangian plane are ridges $F(K,\gRR 1)$, $F(K^{-1},\gRR 2^{-1})$, $F(K^{-1},\gRR 2)$, $F(K,\gRR 1^{-1})$, $F(\gRR 1,Q)$, $F(\gRR 2,Q^{-1})$, $F(\gRR 2^{-1},Q^{-1})$, $F(\gRR 1^{-1},Q)$, $F(\gA 0^{-1}\gR 0)$, $F(\gA 0^{-1},\gR 0^{-1})$, $F(\gA 0,\gR 0^{-1})$ and $F(\gA 0,\gR 0)$.
The proof is done by studying the sign of some of the coordinates and it follows proofs of Proposition 4.8 of \cite{livne}, Proposition 5.3 (end of the proof) of \cite{boadiparker} and Proposition 8.8 of \cite{irene}. 
We will prove the property for the first ridge mentioned. 
The others are done in a similar way.

\begin{prop}
The polyhedra $D$, $\gRR 1^{-1}(D)$, $K^{-1}(D)$ and $\gRR 1^{-1}K^{-1}(D)$ tessellate around the ridge $F(K,\gRR 1)$.
\end{prop}

\begin{proof}
Let us consider points in $D$, $\gRR 1^{-1}(D)$, $K^{-1}(D)$ and $\gRR 1^{-1}K^{-1}(D)$ and record the sign of the values of $\im (z_1)$, $\im(e^{i\phi}z_1)$, $\im (e^{i \theta} z_2)$ and $\im (e^{-i\theta} z_2)$ for them.
They are shown in the following table. 

\begin{longtable}{|c|c|c|c|c|}
\hline
 & $\im(z_1)$ & $\im(e^{i\phi}z_1)$ & $\im (e^{i \theta} z_2)$ & $\im (e^{-i\theta} z_2)$ \\
\hline
$D$ & - &+ & + & - \\
$\gRR 1^{-1}(D)$  & - &+ & - & - \\
$K^{-1}(D)$  & - &- & + & - \\
$\gRR 1^{-1}K^{-1}(D)$  & - &- & - & - \\
\hline
\end{longtable}

The first row can be deduced using the definition of $D$ in terms of the arguments of the coordinates. 
The second row can be deduced by considering that the action of $\gRR 1$ only consists in multiplying the coordinate $z_2$ by $e^{2 i \theta}$. 
The third row can be deduced by the fact that applying $K$ corresponds to first applying $\gA 1$, which multiplies the coordinate $z_1$ by $e^{2 i \phi}$ and then applying $Q$, which relates the $\textbf z$ coordinates to the $\textbf w$ coordinates. 

The ridge $F(K,\gRR 1)$ is defined by $\im(e^{i \phi}z_1)=0$ and $\im(e^{i \theta}z_2)=0$ and in a neighbourhood of the ridge the images considered coincide with the sectors where the values are either positive or negative. 
Combining the information of the table one gets the tessellation as required. 
\end{proof}

\paragraph{Ridges contained in complex lines.}

The ridges contained in complex lines are $F(K,Q)$, $F(K^{-1},Q{-1})$, $F(K, \gR 0^{-1})$, $F(K^{-1},\gR 0)$, $F(\gR 0,\gR 0)$, $F(Q,Q^{-1})$,  $F(\gRR 2,\gA 0^{-1})$, $F(\gRR 1^{-1},\gRR 2)$, $F(\gRR 1^{-1},\gA 0)$, $F(\gRR 1,\gRR 1^{-1})$, $F(\gRR 2,\gRR 2^{-1})$ and $F(\gA 0,\gA 0^{-1})$.
The strategy consists in showing that the polyhedron (and suitable images) cover a sector of amplitude $\psi$ and that the cycle transformation acts on the orthogonal of the complex line as a rotation through angle $\psi$.
Then each power of the cycle transformation covers a sector and since $\psi $ is always $\frac{2\pi}{a}$ with $a $ integer, we cover the whole space around the ridge.
The proofs are similar to the ones of Proposition 4.11 of \cite{livne}, Proposition 5.3 of \cite{boadiparker} (the middle part of the proof) and 8.10 of \cite{irene}. 

The cases of $F(\gRR 1,\gA 0^{-1})$, $F(Q,Q^{-1})$ and $F(K^{-1},\gR 0)$ (and the ones in their cycles) are an exception because the procedure is the same but after applying a suitable change of coordinates. 

The proofs for these cases are along the line of proof of Proposition 4.13 of \cite{livne} and of Proposition 8.11 of \cite{irene}. 
For completeness, we will include the proof of one of these ridges.

\begin{prop}
The polyhedra $D$, $\gA 0 (D)$ and $\gA 0 \gRR 2(D)$ and their images under $\gA 0\gRR 2\gRR 1$ tessellate around the ridge $F(\gRR 1,\gA 0^{-1})$.
\end{prop}

\begin{proof}
The ridge $F(\gRR 1,\gA 0^{-1})$ is contained in $L_{*3}(\alpha,\alpha,\theta,\phi)$.
Remark that $e^{-2i(\theta-\alpha)}\gA 0\gRR 2\gRR 1$ fixes the ridge pointwise and rotates its normal vector $\textbf n_{*3}(\alpha, \alpha,\theta,\phi)$ by $e^{2i(\alpha+\phi-\pi)}$.

The proof consists in changing the coordinates to have a similar situation as for the other ridges contained in a complex line. 
The new coordinates will be in terms of two vectors spanning the complex line and the vector normal to it, since the complex line is the mirror of the transformation $\gA 0\gRR 2\gRR 1$.
More precisely, writing
\begin{multline}
\begin{pmatrix}
x_1 \\ x_2 \\ 1
\end{pmatrix}
=\frac{\sin \phi\sin(\alpha-\theta)-\sin \alpha \sin (\theta+\phi)x_2}{\sin \theta\sin(\alpha+\phi)}
\begin{pmatrix}
0 \\ -1 \\ -1
\end{pmatrix}
+x_1 \begin{pmatrix}
1 \\ 0 \\ 0
\end{pmatrix}\\
+ \frac{1-x_2}{\sin \theta\sin(\alpha+\phi)}
\begin{pmatrix}
0 \\ \sin \phi \sin(\alpha-\theta) \\ \sin (\theta+\phi)\sin \alpha
\end{pmatrix},
\end{multline}
the new coordinates are 
\begin{equation}
\begin{split}
\xi_1&=\frac{\sin \phi\sin(\alpha-\theta)-\sin \alpha \sin (\theta+\phi)x_2}{1-x_2}, \\
\xi_2&=\frac{\sin \theta\sin(\alpha+\phi)x_1}{1-x_2}.
\end{split}
\end{equation}

This means that $\gA 0\gRR 2\gRR 1$ acts on the new coordinates by sending $(\xi_1,\xi_2)$ to the point $(e^{2i(\alpha+\phi-\pi)}\xi_1,\xi_2)$. 
Since the configurations are as in Figure \ref{fig:tessell}, if we prove that $D$, $\gA 0(D)$ and $\gA 0\gRR 2(D)$ cover the sector defined by the argument of $\xi_1$ being between 0 and $2(\alpha+\phi-\pi)$, then the appropriate images under $\gA 0\gRR 2\gRR 1$ will cover a neighbourhood of $F(\gA 0^{-1},\gRR 2)$. 

\begin{figure}
\centering
\includegraphics[width=1\textwidth]{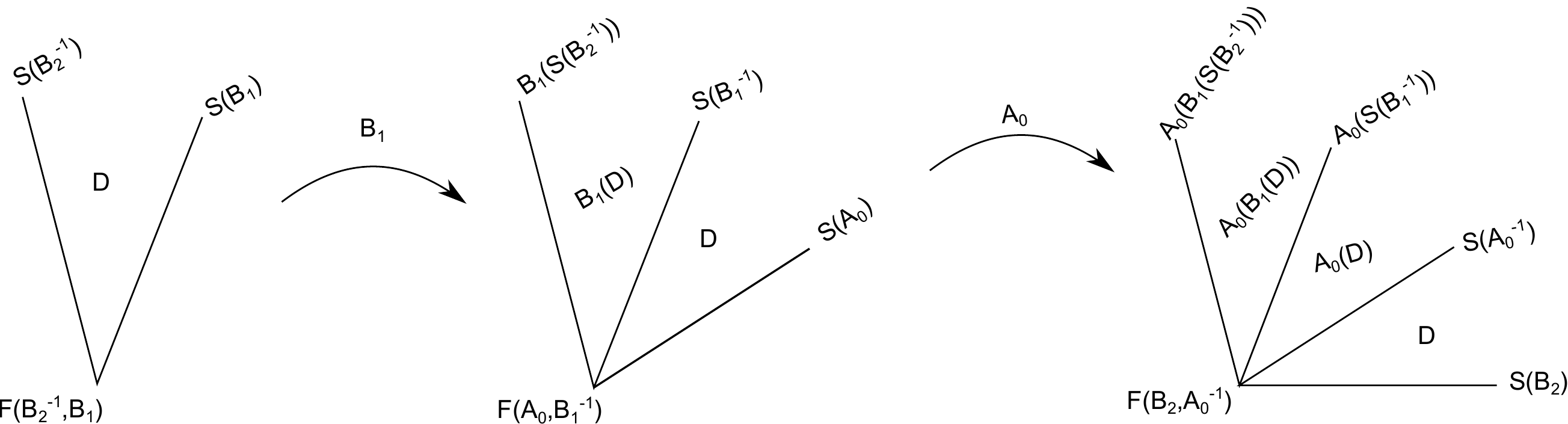}
\begin{quote}\caption{The polyhedra around $F(\gA 0^{-1},\gRR 2)$. \label{fig:tessell}} \end{quote}
\end{figure}

First notice that if we are in $S(\gRR 1)$, then $x_2 \in \R$ and so $\arg \xi_1=0$. 
Moreover, if we take a point in $\textbf z \in S\gRR 1^{-})$, then $z_2=e^{i\theta}u$ with $ u\in\R$ and the coordinate $\xi_1$ of $\gA 0\gRR 2 \textbf z$ is 
\[
\xi_1=e^{2i(\alpha+\phi-\pi)}\frac{\sin(\theta+\phi)u+\sin\phi}{\sin(\alpha-\theta)u-\sin\alpha}
\]
and so $\arg \xi_1=2(\alpha+\phi-\pi)$.

The last thing we need to prove is that such images are disjoint. 
Now the pairs $D,\gA 0D$ and $\gA 0D,\gRR 2\gA 0D$ are disjoint because of tessellation property around $F(\gA 0,\gA 0^{-1})$ and $F(\gRR 2,\gRR 2^{-1})$. 
To prove that $D$ and $\gRR 2\gA 0D$ are disjoint, it is enough to prove that the argument of the coordinate $\xi_1$ of points in $D$ is smaller than $\alpha+\phi-\pi$, while the one of points in $\gRR 2\gA 0D$ is bigger than $\alpha+\phi-\pi$.

If one writes the coordinate $\xi_1$ in terms of the $\textbf{v}$-coordinates, then a point in $S(\gA 0)$ has coordinate $v_1=e^{i\phi'}u$, with $\R\ni u\leq \frac{-\sin(2\alpha)}{\sin(\theta+\phi)}$ by \ref{lemma:sidecombinatorics} and 
\[
\xi_1=e^{i(\alpha+\phi-\pi)}\frac{\sin\phi\sin(\alpha-\theta)(-\sin(2\alpha)-\sin(\theta+\phi)u)}{\sin(\alpha-\theta)u-\sin(\alpha+\phi)v_2+\sin(\alpha-\phi)}.
\]
Then 
\[
\im e^{i(\alpha+\phi-\pi)}\xi_1=\sin\phi\sin(\alpha-\theta)(-\sin(2\alpha)-\sin(\theta+\phi)u)\sin(\alpha+\phi)\im v_2 \geq 0.
\]

Similarly, if we take a point $ \textbf z \in S(\gRR 2^{-1})$, then we have $w_2=e^{-i\theta}v$, with $\R\ni v \leq\frac{\sin\phi}{\sin(\theta+\phi)}$ and the coordinate $\xi_1$ of $\gA 0 \textbf z$ is 
\[
\xi_1=e^{-i(\alpha+\phi-\pi)} \frac{\sin \phi}{\sin\alpha} \cdot
\frac{\sin (\theta+\phi) u-\sin\phi}{\sin(\alpha+\phi)e^{-i\phi}w_1+\sin(\alpha-\theta)u-\sin\alpha}
\]
and 
\[
\im e^{i(\alpha+\phi-\pi)} \xi_1=\frac{\sin\phi}{\sin\alpha} (\sin\phi-\sin(\theta+\phi)u) \sin(\alpha+\phi) \im e^{-i\phi}w_2 \leq 0.
\]

Remark that we are using the fact that $\sin(\alpha+\phi)>0$, which is always the case when the ridge does not collapse (i.e. $l'>0$ and finite). 
\end{proof}

\subsection{The case \texorpdfstring{$k'$}{} negative}\label{sec:kneg}

\begin{figure}[t]
\includegraphics[width=1\textwidth]{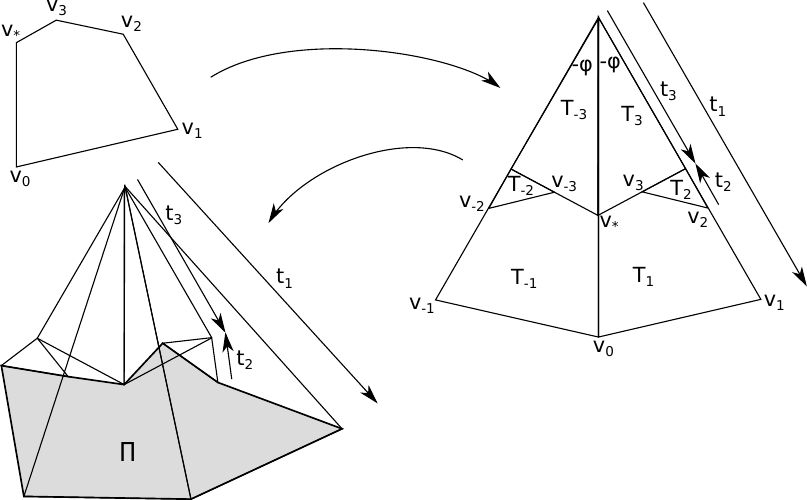}
\begin{quote}\caption{The configuration of triangles when $k'$ is negative. \label{fig:kneg}} \end{quote}
\end{figure}

When $k'$ is negative, after applying $P^{-1}$ to the configuration $(\alpha, \alpha, \theta,\phi)$ we obtain a configuration where the last angle is negative. 
This means that we cannot describe the configuration with the same coordinates and triangles as before. 
This doesn't stop us from doing everything in the same way, up to taking a slightly different configuration of triangles. 
By construction (see Figure \ref{fig:thetapar}), once we developed the cone metric on the plane, $\phi$ was the angle between the line passing through $v_*$ and $v_0$ and the line passing through $v_1$ and $v_2$ on the side of $v_0$ and $v_1$. 
When this angle is negative, we will take $-\phi$ to be the angle between the same two lines, but on the side of $v_2$ and $v_*$ (see Figure \ref{fig:kneg}).

The area of the cone metric is the area of the shaded region $\Pi$. 
Using the coordinates as in the figure, this is 
\[
\area=\frac{-\sin \phi \sin \alpha}{\sin(\alpha-\phi)}|\zzc 1|^2
- \frac{\sin \theta \sin \beta}{\sin(\beta-\theta)}|\zzc 2|^2
-\frac{-\sin\phi \sin \theta}{\sin(\theta+\phi)}|\zzc 3|^2.
\]
Remembering that $-\sin\phi$ is positive, this is still a Hermitian form of the same signature, except that the roles of $\zzc 1$ and $\zzc 3$ are exchanged. 
This makes sense, since now the triangles $T_2$ and $T_3$ are "inside" the triangle $T_1$

When looking at the vertices, this tells us that the we cannot have the line $L_{01}$, since to make $v_0$ and $v_1$ collapse, one should take $x_1=0$ and the whole figure would collapse. 
We will hence have a new vertex $\textbf v_{*23}$ obtained by taking $\zzc 1=\zzc 3=0$ and so by making $v_* \equiv v_2 \equiv v_3$ instead of the three vertices $\textbf y_1, \textbf y_9, \textbf y_{12}$.  
In terms of our polyhedron $D$, this means that $v_0$, $v_7$ and $v_{11}$ collapse to this new point $\textbf v_{*23}$, which is on the boundary (i.e. it makes the area be 0) if $k'$ is infinite. 
All the other vertices remain the same and everything else in the study of the combinatorial structure of the polyhedron can be done in the exact same way. 
In particular, as in Proposition \ref{prop:gencollapse}, we still have that the vertices on $L_{*0}$ collapse to a single vertex if $\pi-\alpha'-\theta' \leq 0$ (i.e. if $d \leq 0$) and the vertices on $L_{*1}$ collapse to a single vertex if $\alpha'-\theta'-\phi' \leq 0$ (i.e. if $l' \leq 0$). 
Remark that the vertices on $L_{*2}$ and $L_{*3}$ never collapse, as $l>0$ in all our cases.
This analysis gives the cases in Theorem \ref{thm:main}.

\addcontentsline{toc}{section}{\refname}
\bibliographystyle{alpha}
\bibliography{biblio}
\nocite{*}

\end{document}